\documentclass[a4paper, intlimits, reqno]{amsart}

\usepackage[english]{babel}
\usepackage[T1]{fontenc}
\usepackage[utf8]{inputenc}

\usepackage{amsmath}
\usepackage{amssymb}
\usepackage{MnSymbol}
\usepackage{amsthm}
\allowdisplaybreaks
\usepackage{amsfonts}
\usepackage{mathrsfs} 
\usepackage{mathtools}
\usepackage[all]{xy}
\usepackage{nicefrac}
\usepackage{enumitem}

\usepackage{multicol}
\usepackage{url}
\usepackage{dsfont}
\usepackage[numbers,sort&compress]{natbib}
\usepackage{doi}
\usepackage{prettyref}
\usepackage{xcolor}
\usepackage{orcidlink}

\newrefformat{defn}{Definition \ref{#1}}
\newrefformat{rem}{Remark \ref{#1}}
\newrefformat{sect}{Section \ref{#1}}
\newrefformat{sub}{Section \ref{#1}}
\newrefformat{prop}{Proposition \ref{#1}}
\newrefformat{thm}{Theorem \ref{#1}}
\newrefformat{cor}{Corollary \ref{#1}}
\newrefformat{ex}{Example \ref{#1}}

\swapnumbers
\newtheoremstyle{dotless}{}{}{\itshape}{}{\bfseries}{}{}{}
\theoremstyle{dotless}
\theoremstyle{plain}
\newtheorem{thm}{Theorem}[section]

\newtheorem{prop}[thm]{Proposition}
\newtheorem{cor}[thm]{Corollary}
\theoremstyle{definition}
\newtheorem{defn}[thm]{Definition}
\newtheorem{rem}[thm]{Remark}
\newtheorem{exa}[thm]{Example}
\newcommand{\N} {\mathbb{N}}

\newcommand{\R} {\mathbb{R}}
\newcommand{\C} {\mathbb{C}}
\newcommand{\K} {\mathbb{K}}

\newcommand{\F} {\mathcal{F}(\Omega)}

\newcommand{\acx} {\operatorname{acx}}

\DeclareMathOperator{\id}{id}
\newcommand{\vertiii}[1]{{\left\vert\kern-0.25ex\left\vert\kern-0.25ex\left\vert #1 
    \right\vert\kern-0.25ex\right\vert\kern-0.25ex\right\vert}}

\makeatletter
\newcommand{\fakephantomsection}{%
  \Hy@GlobalStepCount\Hy@linkcounter%
  \Hy@MakeCurrentHref{\@currenvir.\the\Hy@linkcounter}
  \Hy@raisedlink{\hyper@anchorstart{\@currentHref}\hyper@anchorend}%
}
\makeatother

\begin{document}

\title[On linearisation and existence of preduals]{On linearisation and existence of preduals}
\author[K.~Kruse]{Karsten Kruse\,\orcidlink{0000-0003-1864-4915}}
\address[Karsten Kruse]{University of Twente, Department of Applied Mathematics, P.O. Box 217, 7500 AE Enschede, The Netherlands, and Hamburg University of Technology, Institute of Mathematics, Am Schwarzenberg-Campus~3, 21073 Hamburg, Germany}

\email{k.kruse@utwente.nl}

\subjclass[2020]{Primary 46A08, 46A20 Secondary 46A70, 46B10, 46E10}

\keywords{dual space, predual, linearisation, mixed topology}

\date{\today}
\begin{abstract}
We study the problem of existence of preduals of locally convex Hausdorff spaces. We derive necessary and sufficient 
conditions for the existence of a predual with certain properties of a bornological locally convex Hausdorff space $X$. 
Then we turn to the case that $X=\F$ is a space of scalar-valued functions on a non-empty set $\Omega$ and 
characterise those among them which admit a special predual, namely a \emph{strong linearisation}, 
i.e.~there are a locally convex Hausdorff space $Y$, a map $\delta\colon\Omega\to Y$ and a topological isomorphism 
$T\colon\F\to Y_{b}'$ such that $T(f)\circ \delta= f$ for all $f\in\F$. 
\end{abstract}
\maketitle

\section{Introduction}\label{sect:intro}

The present paper is dedicated to preduals of locally convex Hausdorff spaces, in particular their existence. 
A \emph{predual} of a locally convex Hausdorff space $X$ is a tuple $(Y,\varphi)$ of a locally convex Hausdorff space $Y$ 
and a topological isomorphism $\varphi\colon X\to Y_{b}'$ where $Y_{b}'$ is the strong dual of $Y$. 
The space $X$ is called a \emph{dual space}.
This topic is thoroughly studied in the case of Banach spaces $X$, mostly with regard to isometric preduals. Necessary 
and sufficient conditions for the existence of a Banach predual of a Banach space are due to Dixmier 
\cite[Th\'{e}or\`{e}me 17', p.~1069]{dixmier1948} and were augmented by Waelbroeck \cite[Proposition 1, p.~122]{waelbroeck1966}, 
Ng \cite[Theorem 1, p.~279]{ng1971} and Kaijser \cite[Theorem 1, p.~325]{kaijser1977}.  

All of the preceding results on the existence of a Banach predual of a Banach space involve (relative) compactness with respect 
to an additional locally convex topology on $X$ and were transfered to (ultra)bornological locally convex Hausdorff spaces 
by Mujica \cite[Theorem 1, p.~320--321]{mujica1984a} and then generalised by Bierstedt and Bonet \cite{bierstedt1992}. Namely, 
if $(X,\tau)$ is a bornological locally convex Hausdorff space such that there exists a locally convex Hausdorff topology 
$\widetilde{\tau}$ on $X$ such that
\begin{enumerate}
\item[$\operatorname{(BBC)}$] every $\tau$-bounded subset of $X$ is contained in an absolutely convex $\tau$-bounded $\widetilde{\tau}$-compact set, and 
\item[$\operatorname{(CNC)}$] $\tau$ has a $0$-neighbourhood basis of absolutely convex $\widetilde{\tau}$-closed sets, 
\end{enumerate}
then $(X,\tau)$ has a complete barrelled predual by \cite[1.~Theorem (Mujica), 2.~Corollary, p.~115]{bierstedt1992}. 
Besides noting that $\operatorname{(BBC)}$ and $\operatorname{(CNC)}$ are also necessary for the existence of a complete 
barrelled predual, we derive several other necessary and sufficient conditions for the existence of a complete barrelled predual 
in \prettyref{cor:existence_cb_predual}. In particular, we show that the existence of a \emph{semi-Montel prebidual} $(Y,\varphi)$, 
i.e.~$Y$ is a semi-Montel space and $\varphi\colon (X,\tau)\to (Y_{b}')_{b}'$ a topological isomorphism, such that $Y_{b}'$ is complete 
is a necessary and sufficient condition for the existence of a complete barrelled predual. 
Moreover, we refine such conditions to characterise which Fr\'echet spaces, complete bornological DF-spaces 
and completely normable spaces have complete barrelled DF-preduals, Fr\'echet preduals and Banach preduals, respectively, in 
\prettyref{cor:existence_cbDF_predual}, \prettyref{cor:existence_frechet_predual} and \prettyref{cor:existence_banach_predual}. 
We adapt ideas from the theory of Saks spaces and mixed topologies by Cooper \cite{cooper1978} to achieve this. 

Then we extend these results to obtain necessary and sufficient conditions in \prettyref{thm:existence_scb_linearisation}, 
\prettyref{cor:existence_scbDF_linearisation}, \prettyref{cor:existence_sF_linearisation} and 
\prettyref{cor:existence_sB_linearisation} for the existence of strong linearisations which are 
special preduals of locally convex Hausdorff spaces $\F$ of $\K$-valued functions on a non-empty set $\Omega$ where 
$\K=\R$ or $\C$. We call a triple $(\delta,Y,T)$ of a locally convex Hausdorff space $Y$ over the field $\K$, 
a map $\delta\colon\Omega\to Y$ and a topological isomorphism $T\colon\F\to Y_{b}'$ a \emph{strong linearisation of} $\F$ 
if $T(f)\circ \delta= f$ for all $f\in\F$ (see \cite[p.~683]{carando2004}, \cite[p.~181, 184]{jaramillo2009} 
and \prettyref{prop:equivalent_def_of_linearisation}). In comparison to the necessary and sufficient 
conditions $\operatorname{(BBC)}$ and $\operatorname{(CNC)}$ for the existence of a complete barrelled predual of a 
bornological space $X=\F$ we only have to add that $\widetilde{\tau}$ is finer than the topology of pointwise convergence 
$\tau_{\operatorname{p}}$ to obtain necessary 
and sufficient conditions for the existence of a strong linearisation with complete barrelled $Y$ 
(see \prettyref{thm:existence_scb_linearisation}). 
In \prettyref{cor:existence_scbDF_linearisation}, \prettyref{cor:existence_sF_linearisation} and 
\prettyref{cor:existence_sB_linearisation} we derive corresponding conditions that guarantee that 
$Y$ is a complete barrelled DF-space, Fr\'echet space and completely normable space, respectively. 
Moreover, in \prettyref{thm:existence_scb_cont_linearisation} we give a result on \emph{continuous} strong linearisations, i.e.~where $\delta$ is in addition continuous and $\Omega$ a topological Hausdorff space, 
which generalises \cite[Theorem 2.2, Corollary 2.3, p.~188--189]{jaramillo2009}.
Linearisations are a useful tool since they identify (usually) non-linear functions $f$ with (continuous) linear operators $T(f)$ and thus allow to apply linear functional analysis to non-linear functions. 
We refer the reader who is also interested in the corresponding results of the present paper in the isometric Banach setting, 
where $\F$ and $Y$ are Banach spaces and $T$ an isometry, to \cite{kruse2023c}.

\section{Notions and preliminaries}
\label{sect:notions}

In this short section we recall some basic notions from the theory of locally convex spaces and present some prelimary results 
on dual spaces and their preduals. 
For a locally convex Hausdorff space $X$ over the field $\K\coloneqq\R$ or $\C$ we denote by $X'$ the topological linear dual space 
and by $U^{\circ}$ the \emph{polar set} of a subset $U\subset X$. 
If we want to emphasize the dependency on the locally convex Hausdorff topology $\tau$ of $X$, we write $(X,\tau)$ and 
$(X,\tau)'$ instead of just $X$ and $X'$, respectively. We denote by $\sigma(X',X)$ the topology on $X'$ of uniform convergence 
on finite subsets of $X$, by $\tau_{\operatorname{c}}(X',X)$ the topology on $X'$ of uniform convergence 
on compact subsets of $X$ and by $\beta(X',X)$ the topology on $X'$ of uniform convergence on bounded subsets of $X$. 
Further, we set $X_{b}'\coloneqq (X',\beta(X',X))$. 
Furthermore, we say that a linear map $T\colon X\to Y$ between two locally convex Hausdorff spaces $X$ and $Y$ is \emph{(locally) 
bounded} if it maps bounded sets to bounded sets. 
Moreover, for two locally convex Hausdorff topologies $\tau_{0}$ and $\tau_{1}$ 
on $X$ we write $\tau_{0}\leq\tau_{1}$ if $\tau_{0}$ is coarser than $\tau_{1}$. For a normed space $(X,\|\cdot\|)$ we denote by 
$B_{\|\cdot\|}\coloneqq\{x\in X\;|\;\|x\|\leq 1\}$ the $\|\cdot\|$-closed unit ball of $X$. Further, we write 
$\tau_{\operatorname{co}}$ for the \emph{compact-open topology}, i.e.~the topology of uniform convergence on compact subsets of 
$\Omega$, on the space $\mathcal{C}(\Omega)$ of $\K$-valued continuous functions on a topological Hausdorff space $\Omega$. 
In addition, we write $\tau_{\operatorname{p}}$ for the \emph{topology of pointwise convergence} on the space $\K^{\Omega}$ of $\K$-valued functions on a set $\Omega$. By a slight abuse of notation we also use the symbols $\tau_{\operatorname{co}}$ and 
$\tau_{\operatorname{p}}$ for the relative compact-open topology and the relative topology of pointwise convergence on topological subspaces of $\mathcal{C}(\Omega)$ and $\K^{\Omega}$, respectively. 
For further unexplained notions on the theory of locally convex Hausdorff spaces we refer the reader to \cite{jarchow1981,meisevogt1997,bonet1987}.

\begin{defn}\label{defn:predual}
Let $X$ be a locally convex Hausdorff space. We call $X$ a \emph{dual space} if there are a locally convex Hausdorff space $Y$ and a topological isomorphism $\varphi\colon X\to Y_{b}'$. The tuple $(Y,\varphi)$ is called a \emph{predual} of $X$. 
\end{defn}

In the context of dual Banach spaces the preceding definition of a predual is already given 
e.g.~in \cite[p.~321]{brown1975}. 
If $X$ is a dual space with a quasi-barrelled predual, we may consider this predual as a 
topological subspace of the strong dual of $X$.

\begin{prop}\label{prop:predual_into_dual}
Let $X$ be a dual space with quasi-barrelled predual $(Y,\varphi)$.
Then the map 
\[
\Phi_{\varphi}\colon Y\to X_{b}',\;y\longmapsto[x \mapsto \varphi(x)(y)],
\]
is a topological isomorphism into, i.e.~a topological isomorphism to its range. 
\end{prop}
\begin{proof}
Since $Y$ is quasi-barrelled, the evaluation map $\mathcal{J}_{Y}\colon Y\to (Y_{b}')_{b}'$, $y\longmapsto [y'\mapsto y'(y)]$, 
is a topological isomorphism into by \cite[11.2.2. Proposition, p.~222]{jarchow1981}. 
Furthermore, the map $\varphi$ gives a ono-to-one correspondence 
between the bounded subsets of $X$ and the bounded subsets of $Y_{b}'$. We observe that 
for every bounded set $B\subset X$ it holds that 
\[
 \sup_{x\in B}|\Phi_{\varphi}(y)(x)|
=\sup_{x\in B}|\varphi(x)(y)|
=\sup_{x\in B}|\mathcal{J}_{Y}(y)(\varphi(x))|
=\sup_{y'\in \varphi(B)}|\mathcal{J}_{Y}(y)(y')|
\]
for all $y\in Y$, which proves the claim.
\end{proof}

Next, we come to linearisations of function spaces.

\begin{defn}\label{defn:linearisation}
Let $\F$ be a linear space of $\K$-valued functions on a non-empty set $\Omega$. 
\begin{enumerate}
\item[(a)] We call a triple $(\delta,Y,T)$ of a locally convex Hausdorff space $Y$ over the field $\K$, 
a map $\delta\colon\Omega\to Y$ and an algebraic isomorphism $T\colon\F\to Y'$ a \emph{linearisation of} $\F$ 
if $T(f)\circ \delta= f$ for all $f\in\F$. 
\item[(b)] Let $\Omega$ be a topological Hausdorff space. We call a linearisation $(\delta,Y,T)$ of $\F$ \emph{continuous} 
if $\delta$ is continuous.
\item[(c)] Let $\F$ be a locally convex Hausdorff space. We call a linearisation $(\delta,Y,T)$ of $\F$ \emph{strong} if $T\colon\F\to Y_{b}'$ is a topological isomorphism. 
\item[(d)] We call a (strong) linearisation $(\delta,Y,T)$ of $\F$ a (strong) \emph{complete barrelled (Fr\'echet, DF-, Banach) linearisation} if $Y$ is a complete barrelled (Fr\'echet, DF-, completely normable) space.
\item[(e)] We say that $\F$ \emph{admits a (continuous, strong, complete barrelled, Fr\'echet, DF-, Banach) linearisation} 
if there exists a (continuous, strong, complete barrelled, Fr\'echet, DF-, Banach) linearisation $(\delta,Y,T)$ 
of $\F$.
\end{enumerate}
\end{defn}

Clearly, a strong linearisation $(\delta,Y,T)$ of $\F$ gives us the predual $(Y,T)$ of $\F$. 
\prettyref{defn:linearisation} (c), (d) and (e) are motivated by the definition of a strong Banach linearisation given in 
\cite[p.~184, 187]{jaramillo2009}. 

\begin{exa}\label{ex:linearisation_of_l1}
Let $(\ell^{1},\|\cdot\|_{1})$ denote the Banach space of complex absolutely summable sequences on $\N$, 
$(c_{0},\|\cdot\|_{\infty})$ the Banach space of complex zero sequences 
and $(c,\|\cdot\|_{\infty})$ the Banach space of complex convergent sequences, all three equipped with their usual norms. 

(i) We define the topological isomorphism 
\[
\varphi_{0}\colon (\ell^{1},\|\cdot\|_{1})\to (c_{0},\|\cdot\|_{\infty})_{b}',\;\varphi_{0}(x)(y)\coloneqq \sum_{k=1}^{\infty}x_{k}y_{k}.
\]
Then the tuple $(c_{0},\varphi)$ is a predual of $\ell^{1}$ and we note that
$
x_{n}=\varphi_{0}(x)(e_{n})
$
for all $x\in \ell^{1}$ and $n\in\N$ where $e_{n}$ denotes the $n$-th unit sequence. Setting 
$\delta\colon\N\to c_{0}$, $\delta(n)\coloneqq e_{n}$, we get the strong Banach linearisation 
$(\delta,c_{0},\varphi_{0})$ of $\ell^{1}$. 

(ii) We define the topological isomorphism 
\[
\varphi_{c}\colon (\ell^{1},\|\cdot\|_{1})\to (c,\|\cdot\|_{\infty})_{b}',\;
\varphi_{c}(x)(y)\coloneqq y_{\infty}x_{1}+\sum_{k=1}^{\infty}y_{k}x_{k+1},
\]
where $y_{\infty}\coloneqq\lim_{n\to\infty}y_{n}$ for $y\in c$. 
Then the tuple $(c,\varphi_{c})$ is a predual of $\ell^{1}$ and we claim that there is no 
$\delta_{c}\colon\N\to c$ such that $(\delta_{c},c,\varphi_{c})$ is a linearisation of $\ell^{1}$. 
Suppose the contrary. Then we have 
$
x_{n}=\varphi_{c}(x)(\delta_{c}(n))
$
for all $x\in \ell^{1}$ and $n\in\N$. In particular, we get 
$1=\varphi_{c}(e_{1})(\delta_{c}(1))=\delta_{c}(1)_{\infty}$ and 
$0=\varphi_{c}(e_{m})(\delta_{c}(1))=\delta_{c}(1)_{m-1}$ for all $m\geq 2$, which is a contradiction. 

(iii) We may fix the problem in (ii) by using a different isomorphism. 
We define the topological isomorphism
\[
\psi\colon \;(\ell^{1},\|\cdot\|_{1})\to (c,\|\cdot\|_{\infty})_{b}',\;\psi(x)(y)\coloneqq y_{\infty}x_{1}+\sum_{k=1}^{\infty}(y_{k}-y_{\infty})x_{k+1}.
\]
Then the tuple $(c,\psi)$ is a predual of $\ell^{1}$. To obtain a linearisation of $\ell^{1}$ from this predual, we have to 
to find $\widetilde{\delta}(n)\in c$ such that
$
x_{n}=\psi(x)(\widetilde{\delta}(n))
$ 
for all $x\in\ell^{1}$ and $n\in\N$. Using the unit sequences $e_{m}\in\ell^{1}$, $m\in\N$, 
we see that $\widetilde{\delta}(n)$ has to fulfil
\[
e_{1,n}=\psi(e_{1})(\widetilde{\delta}(n))=\widetilde{\delta}(n)_{\infty},
\]
which implies $\widetilde{\delta}(1)_{\infty}=1$ and $\widetilde{\delta}(n)_{\infty}=0$ for all $n\geq 2$. Further, $\widetilde{\delta}(n)$ has to fulfil for $m\geq 2$
\[
e_{m,n}=\psi(e_{m})(\widetilde{\delta}(n))=\widetilde{\delta}(n)_{m-1}-\widetilde{\delta}(n)_{\infty}.
\]
This yields $\widetilde{\delta}(1)_{m-1}=1$ for all for $m\geq 2$, $\widetilde{\delta}(n)_{m-1}=1$ if $n=m$, $n\geq 2$, and 
$\widetilde{\delta}(n)_{m-1}=0$ if $n\neq m$, $n\geq 2$. So setting $\widetilde{\delta}(1)\coloneqq (1,1,\ldots)\in c$ and 
$\widetilde{\delta}(n)\coloneqq e_{n-1}\in c$ for all $n\geq 2$ and observing that $(e_{m})_{m\in\N}$ 
is a Schauder basis of $\ell^{1}$, we obtain the continuous strong Banach linearisation 
$(\widetilde{\delta},c,\psi)$ of $\ell^{1}$ if $\N$ is equipped with the Hausdorff topology induced by the absolute value $|\cdot|$.
The tuple $(\widetilde{\delta},c)$ is also given in \cite[Example 9, p.~699--700]{carando2004} as a continuous linearisation of $\ell^{1}$ (see \prettyref{prop:equivalent_def_of_linearisation}), 
however, without the information which isomorphism was used to derive $\widetilde{\delta}$ . 
\end{exa}

Our next goal is to compare our definition of a linearisation to the one given e.g.~in 
\cite[p.~181]{jaramillo2009} and \cite[p.~683]{carando2004} and to show that both definitions are equivalent.\footnote{Both definitions are 
equivalent if additional assumptions on $\delta$ are neglected. 
In \cite{jaramillo2009} continuous linearisations are considered and in \cite{carando2004} linearisations such that $\delta$ is 
of the ``same type'' as the functions in $\F$. The constructed linearisation in \cite{carando2004} is continuous, 
see \cite[Theorem 2, p.~689]{carando2004}, and the type of $\delta$ is handled in \cite[Proposition 2, p.~688]{carando2004}.}

\begin{prop}\label{prop:equivalent_def_of_linearisation_tuple}
Let $\F$ be a linear space of $\K$-valued functions on a non-empty set $\Omega$, 
$Y$ a locally convex Hausdorff space over the field $\K$ and $\delta\colon\Omega\to Y$. 
Consider the following conditions for the tuple $(\delta, Y)$.
\begin{enumerate}
\item[(i)] For every continuous linear functional $y'\in Y'$ it holds $y'\circ\delta\in\F$.
\item[(ii)] For every $f\in\F$ there is a unique continuous linear functional $T_{f}\in Y'$ such that 
$T_{f}\circ \delta= f$.
\end{enumerate}
Then the following assertions hold.
\begin{enumerate}
\item[(a)] If condition (ii) is fulfilled, then the map $T\colon\F\to Y'$, $T(f)\coloneqq T_{f}$, is linear and injective. 
\item[(b)] If conditions (i) and (ii) are fulfilled, then the map $T$ is linear and bijective, 
$T^{-1}(y')=y'\circ\delta$ for all $y'\in Y'$, $(\delta,Y,T)$ is a linearisation of $\F$ and 
the span of $\{\delta(x)\;|\;x\in\Omega\}$ dense in $Y$.
\end{enumerate}
\end{prop}
\begin{proof}
(a) By condition (ii) for every $f\in\F$ there is a unique $T_{f}\in Y'$ such that 
$T_{f}\circ\delta=f$. Thus the map $T$ is well-defined. Let $f,g\in\F$ with $T_{f}=T_{g}$, then we get 
\[
f=T_{f}\circ\delta=T_{g}\circ\delta=g,
\]
which implies that $T$ is injective. Next, we turn to linearity. Let $f,g\in\F$ and $\lambda\in\K$. Then we have 
$(T_{f}+T_{g})\circ\delta=f+g=(T_{f+g})\circ\delta$ and $T_{\lambda f}\circ\delta=\lambda f=(\lambda T_{f})\circ\delta$. Due to uniqueness we get that $T$ is linear. 

(b) First, we show that $T$ is surjective. Let $y'\in Y'$. Then $f_{y'}\coloneqq y'\circ\delta\in\F$ by condition (i), 
$T(f_{y'})\in Y'$ and
\[
T(f_{y'})\circ\delta=f_{y'}=y'\circ\delta
\]
by condition (ii). The uniqueness of the functional in $Y'$ in condition (ii) implies that $T(y'\circ\delta)=T(f_{y'})=y'$. 
Hence $T$ is surjective, so bijective by part (a), and $T^{-1}(y')=y'\circ\delta$ for all $y'\in Y'$ 
and $(\delta,Y,T)$ is also a linearisation of $\F$.

Second, suppose that the span of $\{\delta(x)\;|\;x\in\Omega\}$ is not dense in $Y$. 
Then there is $u\in Y'$, $u\neq 0$, such that $u(\delta(x))=0$ for all $x\in\Omega$ by the bipolar theorem. 
Since $T$ is bijective, there is $f^{u}\in\F$, $f^{u}\neq 0$, such that $T(f^{u})=u$. It follows that 
$f^{u}(x)=(T(f^{u})\circ\delta)(x)=(u\circ\delta)(x)=0$ for all $x\in\Omega$ by condition (ii), which is a 
contradiction. 
\end{proof}

In \cite[p.~181]{jaramillo2009} a linearisation of $\F$ is defined as a tuple $(\delta,Y)$ 
that fulfils conditions (i) and (ii) of \prettyref{prop:equivalent_def_of_linearisation_tuple} (if we ignore the assumption 
that $\delta$ is continuous in \cite[p.~181]{jaramillo2009}).
The following result shows that our definition of a linearisation is equivalent to the one 
in \cite[p.~181]{jaramillo2009}.

\begin{prop}\label{prop:equivalent_def_of_linearisation}
Let $\F$ be a linear space of $\K$-valued functions on a non-empty set $\Omega$, 
$Y$ a locally convex Hausdorff space over the field $\K$ and $\delta\colon\Omega\to Y$. 
Then the following assertions are equivalent.
\begin{enumerate}
\item[(a)] $(\delta,Y)$ fulfils conditions (i) and (ii) of \prettyref{prop:equivalent_def_of_linearisation_tuple}.
\item[(b)] There is a (unique) algebraic isomorphism $T\colon \F\to Y'$ such that $(\delta,Y,T)$ is a linearisation of $\F$. 
\end{enumerate}  
\end{prop}
\begin{proof}
(a)$\Rightarrow$(b) The existence of the algebraic isomorphism $T$ in (b) follows from 
\prettyref{prop:equivalent_def_of_linearisation_tuple} (b). Let $\widetilde{T}$ be another algebraic isomorphism such that 
$(\delta,Y,\widetilde{T})$ is a linearisation of $\F$. Let $f\in \F$. Then we have 
$
\widetilde{T}(f)(\delta(x))=f=T(f)(\delta(x))
$ 
for all $x\in\Omega$. So $\widetilde{T}(f)$ and $T(f)$ coincide on 
the span of $\{\delta(x)\;|\;x\in\Omega\}$, which is dense in $Y$ by \prettyref{prop:equivalent_def_of_linearisation_tuple} (b). 
Hence the continuity of $\widetilde{T}(f)$ and $T(f)$ implies that $\widetilde{T}(f)=T(f)$ on $Y$, which settles the uniqueness.

(b)$\Rightarrow$(a) Let $T\colon \F\to Y'$ be an algebraic isomorphism such that $(\delta,Y,T)$ is a linearisation of $\F$. 
Condition (i) follows from the surjectivity of $T$ and $T(f)\circ\delta=f$ for all $f\in\F$. 
The second part of the proof of \prettyref{prop:equivalent_def_of_linearisation_tuple} (b) shows that 
the span of $\{\delta(x)\;|\;x\in\Omega\}$ is dense in $Y$ since $T$ is bijective and $T(f)\circ\delta=f$ for all $f\in\F$. 
Let $T_{f}\in Y'$ such that $T_ {f}\circ\delta=f$ for all $f\in\F$. Then we obtain as above that $T_{f}=T(f)$ on $Y$. 
Hence condition (ii) is fulfilled as well. 
\end{proof}

Next, we give a precise characterisation of the surjectivity of the map $T$ from a triple $(\delta,Y,T)$ that is almost a linearisation.
 
\begin{prop}\label{prop:equivalent_def_of_linearisation_dense}
Let $\F$ be a linear space of $\K$-valued functions on a non-empty set $\Omega$, 
$Y$ a locally convex Hausdorff space over the field $\K$, $\delta\colon\Omega\to Y$ a map 
and $T\colon\F\to Y'$ a linear injective map
such that $T(f)\circ\delta=f$ for all $f\in\F$. Then the following assertions are equivalent.
\begin{enumerate}
\item[(a)] $T$ is surjective.
\item[(b)] The tuple $(\delta,Y)$ fulfils condition (i) of \prettyref{prop:equivalent_def_of_linearisation_tuple} 
and the span of $\{\delta(x)\;|\;x\in\Omega\}$ is dense in $Y$.
\end{enumerate}  
\end{prop}
\begin{proof}
(a)$\Rightarrow$(b) If $T$ is surjective, then $(\delta,Y,T)$ is a linearisation of $\F$. 
The proof of the implication (b)$\Rightarrow$(a) of \prettyref{prop:equivalent_def_of_linearisation} 
shows that the span of $\{\delta(x)\;|\;x\in\Omega\}$ is dense in $Y$ and condition (i) is fulfilled.

(b)$\Rightarrow$(a) Let $y'\in Y'$. Then $y'\circ\delta\in\F$ by condition (i), $T(y'\circ\delta)\in Y'$ and
\[
T(y'\circ\delta)(\delta(x))=(y'\circ\delta)(x)=y'(\delta(x))
\]
for all $x\in\Omega$. Thus the continuous linear functionals $T(y'\circ\delta)$ and $y'$ coincide on 
a dense subspace of $Y$ and so on the whole space $Y$. Hence $T$ is surjective.
\end{proof}

\section{Existence of preduals and prebiduals}
\label{sect:existence}

In this section we study necessary and sufficient conditions that guarantee the existence of a predual. 
We recall the conditions $\operatorname{(BBC)}$ and $\operatorname{(CNC)}$ from \cite[p.~114]{bierstedt1992} and \cite[2.~Corollary, p.~115]{bierstedt1992}, which were introduced in a slightly less general form in \cite[Theorem 1, p.~320--321]{mujica1984a}. 
Further, we introduce the condition $\operatorname{(BBCl)}$, 
which is a generalisation of the compatibility condition in \cite[p.~6]{cooper1978}. 
 
\begin{defn}
Let $(X,\tau)$ be a locally convex Hausdorff space. 
\begin{enumerate}
\item[(a)] We say that $(X,\tau)$ satisfies condition $\operatorname{(BBCl)}$ 
if there exists a locally convex Hausdorff topology $\widetilde{\tau}$ on $X$ such that every $\tau$-bounded subset of $X$ 
is contained in an absolutely convex $\tau$-bounded $\widetilde{\tau}$-closed set. 
\item[(b)] We say that $(X,\tau)$ satisfies condition $\operatorname{(BBC)}$ 
if there exists a locally convex Hausdorff topology $\widetilde{\tau}$ on $X$ such that every $\tau$-bounded subset of $X$ 
is contained in an absolutely convex $\tau$-bounded $\widetilde{\tau}$-compact set. 
\item[(c)] We say that $(X,\tau)$ satisfies condition $\operatorname{(CNC)}$ if there exists a locally convex Hausdorff topology 
$\widetilde{\tau}$ on $X$ such that $\tau$ has a $0$-neighbourhood basis $\mathcal{U}_{0}$ 
of absolutely convex $\widetilde{\tau}$-closed sets. 
\end{enumerate}
If we want to emphasize the dependency on $\widetilde{\tau}$ we say that $(X,\tau)$ satisfies $\operatorname{(BBC)}$, 
$\operatorname{(BBCl)}$ resp.~$\operatorname{(CNC)}$ for $\widetilde{\tau}$. We say $(X,\tau)$ satisfies $\operatorname{(BBC)}$ 
(or $\operatorname{(BBCl)})$ and $\operatorname{(CNC)}$ for $\widetilde{\tau}$ if it satisfies both conditions for the same 
$\widetilde{\tau}$.
\end{defn}

Obviously, $\operatorname{(BBC)}$ implies $\operatorname{(BBCl)}$. Let us collect some other useful observations concerning 
these conditions. The observations (b), (c), (e) and (f) of \prettyref{rem:BBC_CNC} below are taken from 
\cite[p.~114, 116]{bierstedt1992}, (a) is clearly valid by the definition of $\operatorname{(BBCl)}$, (d) follows from (b) and 
the remarks above \cite[Chap.~3, \S9, Proposition 2, p.~231]{horvath1966}, and (f) follows from the definition of 
$\operatorname{(CNC)}$, \cite[8.2.5 Proposition, p.~148]{jarchow1981} and the Mackey--Arens theorem. 

\begin{rem}\label{rem:BBC_CNC}
Let $(X,\tau)$ be a bornological locally convex Hausdorff space and $\mathcal{B}$ be the family of $\tau$-bounded sets.
\begin{enumerate}
\item[(a)] Let $\widetilde{\tau}$ be a locally convex Hausdorff topology on $X$. Then $(X,\tau)$ satisfies $\operatorname{(BBCl)}$ for $\widetilde{\tau}$ if and only if $\mathcal{B}$ has a basis $\mathcal{B}_{0}$ of absolutely convex $\widetilde{\tau}$-closed sets. 
\item[(b)] Let $\widetilde{\tau}$ be a locally convex Hausdorff topology on $X$. Then $(X,\tau)$ satisfies $\operatorname{(BBC)}$ for $\widetilde{\tau}$ if and only if $\mathcal{B}$ has a basis $\mathcal{B}_{1}$ of absolutely convex $\widetilde{\tau}$-compact sets. 
\item[(c)] If $(X,\tau)$ satisfies $\operatorname{(BBC)}$ for some $\widetilde{\tau}$, 
then $\widetilde{\tau}\leq\tau$ and $(X,\tau)$ is ultrabornological. 
\item[(d)] If $(X,\tau)$ satisfies $\operatorname{(BBC)}$ for some $\widetilde{\tau}$, 
then $(X,\tau)$ satisfies $\operatorname{(BBC)}$ for all locally convex Hausdorff topologies $\widetilde{\tau}_{0}$ on $X$ such 
that $\widetilde{\tau}_{0}\leq \widetilde{\tau}$ since $\widetilde{\tau}_{0}$ and $\widetilde{\tau}$ coincide on all 
$B\in\mathcal{B}_{1}$.
\item[(e)] If $(X,\tau)$ satisfies $\operatorname{(BBC)}$ and $\operatorname{(CNC)}$ for some 
$\widetilde{\tau}$, then $(X,\tau)$ is quasi-complete.
\item[(f)] If $(X,\tau)$ satisfies $\operatorname{(CNC)}$ for some $\widetilde{\tau}$, 
then $(X,\tau)$ satisfies $\operatorname{(CNC)}$ for all locally convex Hausdorff topologies $\widetilde{\tau}_{0}$ on $X$ such that 
$\sigma(X,(X,\widetilde{\tau})')\leq \widetilde{\tau}_{0}\leq \mu(X,(X,\widetilde{\tau})')$.
\end{enumerate}
\end{rem}

Let $\Omega$ be a non-empty topological Hausdorff space. 
We call $\mathcal{V}$ a \emph{directed family of continuous weights} 
if $\mathcal{V}$ is a family of continuous functions $v\colon\Omega\to[0,\infty)$ such that for every $v_{1},v_{2}\in \mathcal{V}$ there are 
$C\geq 0$ and $v_{0}\in \mathcal{V}$ with $\max(v_{1},v_{2})\leq Cv_{0}$ on $\Omega$. 
We call a directed family of continuous weights $\mathcal{V}$ \emph{point-detecting} if for every $x\in\Omega$ there is $v\in \mathcal{V}$ 
such that $v(x)>0$. For an open set $\Omega\subset\R^{d}$ we denote by $\mathcal{C}^{\infty}(\Omega)$ the space of $\K$-valued infinitely continuously partially differentiable functions on $\Omega$. 
The next example is a slight generalisation of \cite[p.~34]{bonet2002} where the weighted space 
$\mathcal{HV}(\Omega)$ of holomorphic functions on an open connected set $\Omega\subset\C^{d}$ is considered and $\mathcal{V}$ is 
a point-detecting Nachbin family of continuous weights.

\begin{exa}\label{ex:borno_frechet_hypo}
Let $\Omega\subset\R^{d}$ be open, $P(\partial)$ a hypoelliptic linear partial differential operator on $\mathcal{C}^{\infty}(\Omega)$ 
and $\mathcal{V}$ a point-detecting directed family of continuous weights. We define the space
\[
\mathcal{C}_{P}\mathcal{V}(\Omega)\coloneqq\{f\in\mathcal{C}_{P}(\Omega)\;|\;\forall\;v\in \mathcal{V}:\;\|f\|_{v}\coloneqq\sup_{x\in\Omega}|f(x)|v(x)<\infty\},
\] 
where $\mathcal{C}_{P}(\Omega)\coloneqq\{f\in\mathcal{C}^{\infty}(\Omega)\;|\;f\in\ker{P(\partial)}\}$, and equip 
$\mathcal{C}_{P}\mathcal{V}(\Omega)$ with the locally convex Hausdorff topology $\tau_{\mathcal{V}}$ induced by the seminorms $(\|\cdot\|_{v})_{v\in \mathcal{V}}$. 
The space $(\mathcal{C}_{P}\mathcal{V}(\Omega),\tau_{\mathcal{V}})$ is complete and 
$\tau_{\operatorname{co}}\coloneqq{\tau_{\operatorname{co}}}_{\mid\mathcal{C}_{P}\mathcal{V}(\Omega)}\leq \tau_{\mathcal{V}}$. Further, 
the absolutely convex sets 
\[
U_{v}\coloneqq \{f\in\mathcal{C}_{P}\mathcal{V}(\Omega)\;|\;\|f\|_{v}\leq 1\},\quad v\in \mathcal{V},
\]
are $\tau_{\operatorname{co}}$-closed and form a $0$-neighbourhood basis of $\tau_{\mathcal{V}}$ (cf.~\cite[p.~34]{bonet2002} 
where $\mathcal{V}$ is a Nachbin family). 
Therefore $(\mathcal{C}_{P}\mathcal{V}(\Omega),\tau_{\mathcal{V}})$ satisfies $\operatorname{(CNC)}$ for $\tau_{\operatorname{co}}$.
If follows similarly to \cite[p.~123]{bierstedt1992} or the proof of \cite[Proposition 1.2 (c), p.~274--275]{bierstedt1993} 
that every $\tau_{\mathcal{V}}$-bounded set $B$ is contained in an absolutely convex 
$\tau_{\operatorname{co}}$-closed set $B_{1}$. Since $(\mathcal{C}_{P}(\Omega),\tau_{\operatorname{co}})$ is a 
Fr\'echet--Schwartz space (see e.g.~\cite[p.~690]{frerick2009}), the set $B_{1}$ is also $\tau_{\operatorname{co}}$-compact. 
Therefore $(\mathcal{C}_{P}\mathcal{V}(\Omega),\tau_{\mathcal{V}})$ satisfies $\operatorname{(BBC)}$ for $\tau_{\operatorname{co}}$. 
If $\mathcal{V}=(v_{n})_{n\in\N}$ is countable and \emph{increasing}, i.e.~$v_{n}\leq v_{n+1}$ for all $n\in\N$, then 
$(\mathcal{C}_{P}\mathcal{V}(\Omega),\tau_{\mathcal{V}})$ is a Fr\'echet space, in particular bornological. If $\mathcal{V}=\{v\}$, 
we set $\mathcal{C}_{P}v(\Omega)\coloneqq\mathcal{C}_{P}\mathcal{V}(\Omega)$ and note that $(\mathcal{C}_{P}v(\Omega),\|\cdot\|_{v})$ 
is a Banach space.
\end{exa}

\begin{prop}\label{prop:BBC_CNC_normable}
Let $(X,\tau)$ be a normable locally convex Hausdorff space. Then the following assertions are equivalent.
\begin{enumerate}
\item[(a)] $(X,\tau)$ satisfies $\operatorname{(BBC)}$ and $\operatorname{(CNC)}$ for some 
$\widetilde{\tau}$. 
\item[(b)] $(X,\tau)$ is quasi-complete and satisfies $\operatorname{(BBC)}$ for some $\widetilde{\tau}$.
\item[(c)] There are a norm $\vertiii{\cdot}$ on $X$ which induces $\tau$, and a locally convex Hausdorff topology $\widetilde{\tau}$ such that $B_{\vertiii{\cdot}}$ $\widetilde{\tau}$-compact.
\end{enumerate}
\end{prop}
\begin{proof}
First, we note that normable spaces are bornological by \cite[Proposition 24.10, p.~282]{meisevogt1997}.

(a)$\Rightarrow$(b) Using \prettyref{rem:BBC_CNC} (e), this implication is obvious. 

(b)$\Rightarrow$(c) Since $(X,\tau)$ is quasi-complete and normable, there is a norm $\|\cdot\|$ on $X$ such that the identity map 
$\operatorname{id}\colon (X,\tau)\to (X,\|\cdot\|)$ is a topological isomorphism and $(X,\|\cdot\|)$ quasi-complete, thus complete. Hence a subset of $X$ is $\tau$-bounded if and only if it is $\|\cdot\|$-bounded. Since $(X,\tau)$ satisfies $\operatorname{(BBC)}$ for some $\widetilde{\tau}$, there is an absolutely convex $\|\cdot\|$-bounded $\widetilde{\tau}$-compact subset $B$ of $X$ such that $B_{\|\cdot\|}\subset B$. This implies that the Minkowski functional (or gauge) of $B$ given by
\[
\vertiii{x}\coloneqq\inf\{t>0\;|\;x\in tB\},\quad x\in X,
\]
defines a norm on $X$ such that there is $C\geq 0$ with $\|x\|\leq C\vertiii{x}$ for all $x\in X$ 
by \cite[p.~151]{jarchow1981}. Due to the $\widetilde{\tau}$-compactness of $B$ 
and \prettyref{rem:BBC_CNC} (c) the set $B$ is $\tau$-closed and thus $\|\cdot\|$-closed. Further, the completeness of $(X,\|\cdot\|)$ yields that $B$ is sequentially $\|\cdot\|$-complete and hence a Banach disk by \cite[Corollary 23.14, p.~268]{meisevogt1997}. Therefore $(X,\vertiii{\cdot})$ is complete and 
$\id\colon (X,\vertiii{\cdot})\to (X,\|\cdot\|)$ a topological isomorphism by \cite[Open mapping theorem 24.30, p.~289]{meisevogt1997}. Since $B$ is $\|\cdot\|$-closed, it is also $\vertiii{\cdot}$-closed 
and hence $B_{\vertiii{\cdot}}=B$ by \cite[Remark 6.6 (b), p.~47]{meisevogt1997}. 
We conclude that $B_{\vertiii{\cdot}}$ is $\widetilde{\tau}$-compact.

(c)$\Rightarrow$(a) We set $\mathcal{B}_{1}\coloneqq\mathcal{U}_{0}\coloneqq\{tB_{\vertiii{\cdot}}\;|\;t>0\}$. Then $\mathcal{B}_{1}$ is a basis of $\tau$-bounded sets which are absolutely convex and $\widetilde{\tau}$-compact, implying that $(X,\tau)$ satisfies $\operatorname{(BBC)}$ for $\widetilde{\tau}$ by \prettyref{rem:BBC_CNC} (b). 
$\mathcal{B}_{1}=\mathcal{U}_{0}$ is also a $0$-neighbourhood basis w.r.t.~$\tau$
of absolutely convex $\widetilde{\tau}$-closed sets, yielding that $(X,\tau)$ satisfies $\operatorname{(CNC)}$ for $\widetilde{\tau}$.
\end{proof}

Similarly, we get by analysing the proof above a corresponding result with $\operatorname{(BBCl)}$ instead of $\operatorname{(BBC)}$.

\begin{prop}\label{prop:BBCl_CNC_normable}
Let $(X,\tau)$ be a completely normable locally convex Hausdorff space. Then the following assertions are equivalent.
\begin{enumerate}
\item[(a)] $(X,\tau)$ satisfies $\operatorname{(BBCl)}$ and $\operatorname{(CNC)}$ for some 
$\widetilde{\tau}\leq\tau$. 
\item[(b)] $(X,\tau)$ is complete and satisfies $\operatorname{(BBCl)}$ for some $\widetilde{\tau}\leq\tau$.
\item[(c)] There are a norm $\vertiii{\cdot}$ on $X$ which induces $\tau$, and a locally convex Hausdorff topology $\widetilde{\tau}\leq\tau$ such that $(X,\vertiii{\cdot})$ is complete and $B_{\vertiii{\cdot}}$ $\widetilde{\tau}$-closed.
\end{enumerate}
\end{prop}

Now, we recall the candidate for a predual from e.g.~\cite{bierstedt1992,mujica1984a,ng1971}. 
Let $(X,\tau)$ be a locally convex Hausdorff space, $\mathcal{B}$ the family of $\tau$-bounded sets and $\widetilde{\tau}$ 
another locally convex Hausdorff topology on $X$. We denote by $X^{\star}$ the algebraic dual space of $X$ and define 
\[
 X_{\mathcal{B},\widetilde{\tau}}'
\coloneqq\{x^{\star}\in X^{\star}\;|\;x^{\star}_{\mid B}\text{ is }\widetilde{\tau}\text{-continuous for all }B\in\mathcal{B}\}
\]
and observe that $(X,\widetilde{\tau})'\subset X_{\mathcal{B},\widetilde{\tau}}'$ as linear spaces. 
We equip $X_{\mathcal{B},\widetilde{\tau}}'$ with the topology 
$\beta\coloneqq\beta_{\mathcal{B},\widetilde{\tau}}\coloneqq\beta(X_{\mathcal{B},\widetilde{\tau}}',(X,\tau))$ 
of uniform convergence on the $\tau$-bounded subsets of $X$. 
If $(X,\tau)$ is bornological and satisfies $\operatorname{(BBC)}$ for $\widetilde{\tau}$, then we have $\widetilde{\tau}\leq\tau$ 
by \prettyref{rem:BBC_CNC} (c), 
\[
X_{\mathcal{B},\widetilde{\tau}}'=\{x^{\star}\in X^{\star}\;|\;x^{\star}_{\mid B}\text{ is }\widetilde{\tau}\text{-continuous for all }B\in\mathcal{B}_{1}\}
\]
with $\mathcal{B}_{1}$ from \prettyref{rem:BBC_CNC} (b), $X_{\mathcal{B},\widetilde{\tau}}' \subset (X,\tau)'$ and 
\[
 \beta
=\beta_{\mathcal{B},\widetilde{\tau}}
=\beta(X_{\mathcal{B},\widetilde{\tau}}',(X,\tau))
=\beta((X,\tau)',(X,\tau))_{\mid X_{\mathcal{B},\widetilde{\tau}}'}
=\widetilde{\beta}_{\mathcal{B}_{1},\widetilde{\tau}}
\]
by \cite[p.~115]{bierstedt1992} where $\widetilde{\beta}_{\mathcal{B}_{1},\widetilde{\tau}}$ 
denotes the topology on $X_{\mathcal{B},\widetilde{\tau}}'$ of uniform convergence on the sets $B\in\mathcal{B}_{1}$.

\begin{rem}\label{rem:predual_independent}
Let $(X,\tau)$ be a bornological locally convex Hausdorff space satisfying $\operatorname{(BBC)}$ for some $\widetilde{\tau}$ 
and $\mathcal{B}$ the family of $\tau$-bounded sets. If $\widetilde{\tau}_{0}$ is a locally convex Hausdorff topology on $X$ 
such that $\widetilde{\tau}_{0}\leq \widetilde{\tau}$, then $X_{\mathcal{B},\widetilde{\tau}}'=X_{\mathcal{B},\widetilde{\tau}_{0}}'$ 
and $\beta_{\mathcal{B},\widetilde{\tau}}=\beta_{\mathcal{B},\widetilde{\tau}_{0}}$. 
Indeed, this observation follows directly from the considerations above and \prettyref{rem:BBC_CNC} (d). 
\end{rem}

\begin{prop}\label{prop:predual_complete}
Let $(X,\tau)$ be a bornological locally convex Hausdorff space satisfying $\operatorname{(BBC)}$ for some $\widetilde{\tau}$ 
and $\mathcal{B}$ the family of $\tau$-bounded sets. Then the following assertions hold.
\begin{enumerate}
\item[(a)] $X_{\mathcal{B},\widetilde{\tau}}'$ is a closed subspace of the complete space $(X,\tau)_{b}'$. 
In particular, $(X_{\mathcal{B},\widetilde{\tau}}',\beta)$ is complete. 
\item[(b)] If $(X,\tau)$ is a DF-space, then $(X_{\mathcal{B},\widetilde{\tau}}',\beta)$ is 
a Fr\'echet space. 
\item[(c)] If $(X,\tau)$ is normable, then $(X_{\mathcal{B},\widetilde{\tau}}',\beta)$ is 
completely normable. 
\end{enumerate}
\end{prop}
\begin{proof}
(a) This follows from the proof of \cite[1.~Theorem (Mujica), p.~115]{bierstedt1992}.

(b) Since $(X,\tau)$ is a DF-space, its strong dual $(X,\tau)_{b}'$ is a Fr\'echet space by \cite[12.4.2 Theorem, p.~258]{jarchow1981}. It follows from part (a) that $X_{\mathcal{B},\widetilde{\tau}}'$ is a closed subspace of 
$(X,\tau)_{b}'$ and thus $(X_{\mathcal{B},\widetilde{\tau}}',\beta)$ a Fr\'echet space as well. 

(c) Since $(X,\tau)$ is a normable space, its strong dual $(X,\tau)_{b}'$ is a completely normable space. The rest follows as in (b).
\end{proof}

The conditions $\operatorname{(BBC)}$ and $\operatorname{(CNC)}$ for some 
$\widetilde{\tau}$ guarantee that $X_{\mathcal{B},\widetilde{\tau}}'$ (equipped with a suitable topological isomorphism) 
is a complete barrelled predual of a bornological locally convex Hausdorff space $(X,\tau)$.

\begin{thm}[{\cite[1.~Theorem (Mujica), 2.~Corollary, p.~115]{bierstedt1992}}]\label{thm:general_dixmier_ng}
Let $(X,\tau)$ be a bornological locally convex Hausdorff space satisfying $\operatorname{(BBC)}$ and $\operatorname{(CNC)}$ for some 
$\widetilde{\tau}$ and $\mathcal{B}$ the family of $\tau$-bounded sets. 
Then $(X_{\mathcal{B},\widetilde{\tau}}',\beta)$ is a complete barrelled locally convex Hausdorff space 
and the evaluation map 
\[
\mathcal{I}\colon(X,\tau)\to (X_{\mathcal{B},\widetilde{\tau}}',\beta)_{b}',\; x\longmapsto [x'\mapsto x'(x)], 
\]
is a topological isomorphism. In particular, $(X,\tau)$ is a dual space with complete barrelled predual 
$(X_{\mathcal{B},\widetilde{\tau}}',\mathcal{I})$.
\end{thm}

Next, we show that under suitable assumptions $(X,\tau)$ is also a bidual of a locally convex Hausdorff space, 
more precisely topologically isomorphic to a bidual. In contrast to the strategy in 
\cite[Section 2, p.~118--122]{bierstedt1992}, we do not achieve this by replacing $(X_{\mathcal{B},\widetilde{\tau}}',\beta)$ 
by the strong dual $H_{b}'$ of a suitable topological subspace $H$ of $(X,\tau)$, but by finding a locally convex Hausdorff topology 
$\widetilde{\gamma}$ on $X$ such that $(X,\widetilde{\gamma})_{b}'=(X_{\mathcal{B},\widetilde{\tau}}',\beta)$ 
under suitable assumptions. This is the strategy that is also used in \cite[1.3 Proposition, p.~276]{bierstedt1993} 
and \cite[4.7 Proposition (b), p.~877--878]{mujica1991} in the case that $X$ is a weighted space of holomorphic functions. 
This needs a bit of preparation.

\begin{defn}
Let $X$ be a locally convex Hausdorff space. We call $X$ a \emph{bidual space} if there are a locally convex Hausdorff space $Y$ 
and a topological isomorphism $\varphi\colon X\to (Y_{b}')_{b}'$. The tuple $(Y,\varphi)$ is called a \emph{prebidual} of $X$. 
\end{defn}

We note that if $X$ is a bidual space with prebidual $(Y,\varphi)$, then $X$ is a dual space with predual $(Y_{b}',\varphi)$. 
In particular, every reflexive locally convex Hausdorff space is a bidual space. 

\begin{prop}\label{prop:bounded_equicont}
Let $(X,\tau)$ be a locally convex Hausdorff space, $B\subset X$ and 
$\mathcal{J}\colon (X,\tau)\to ((X,\tau)_{b}')'$, $x\longmapsto[x'\mapsto x'(x)]$, the canonical linear injection. 
Then $B$ is $\tau$-bounded if and only if $\mathcal{J}(B)\subset ((X,\tau)_{b}')'$ is equicontinuous.
\end{prop} 
\begin{proof}
$\Rightarrow$ Let $B$ be $\tau$-bounded. Then we have
\[
|\mathcal{J}(x)(x')|=|x'(x)|\leq \sup_{z\in B}|x'(z)|
\]
for all $x\in B$ and $x'\in (X,\tau)'$, meaning that $\mathcal{J}(B)$ is equicontinuous. 

$\Leftarrow$ Let $\mathcal{J}(B)$ be equicontinuous. Then there are a $\tau$-bounded set $\widetilde{B}\subset X$ and $C\geq 0$ 
such that 
\[
|x'(x)|=|\mathcal{J}(x)(x')|\leq C \sup_{z\in \widetilde{B}}|x'(z)|<\infty
\]
for all $x\in B$ and $x'\in (X,\tau)'$. This yields that $B$ is $\sigma(X,(X,\tau)')$-bounded and thus $\tau$-bounded 
by the Mackey theorem.
\end{proof}

\begin{prop}\label{prop:bornological_comp_top}
Let $(X,\tau)$ be a bornological locally convex Hausdorff space and $\widetilde{\tau}$ a locally convex Hausdorff topology on $X$. 
Then $\widetilde{\tau}\leq\tau$ if and only if every $\tau$-bounded set is $\widetilde{\tau}$-bounded.
\end{prop} 
\begin{proof}
The implication $\Rightarrow$ is obvious. Let us turn to $\Leftarrow$. We note that the identity map 
$\id\colon (X,\tau)\to(X,\widetilde{\tau})$ is bounded since every $\tau$-bounded set is $\widetilde{\tau}$-bounded. 
This yields that $\id$ is continuous by \cite[Proposition 24.13, p.~283]{meisevogt1997} as $(X,\tau)$ is bornological.
\end{proof}

\begin{defn}
Let $(X,\tau)$ be a locally convex Hausdorff space and $\widetilde{\tau}$ a locally convex Hausdorff topology on $X$. 
Let $\widetilde{\gamma}\coloneqq\widetilde{\gamma}(\tau,\widetilde{\tau})$ denote the finest 
locally convex Hausdorff topology on $X$ which coincides with $\widetilde{\tau}$ on $\tau$-bounded sets. 
We say that $(X,\widetilde{\gamma})$ satisfies $\operatorname{(B\tau B}\widetilde{\gamma})$ if a subset of $X$ is $\tau$-bounded if 
and only if it is $\widetilde{\gamma}$-bounded.
\end{defn}

For a linear space $X$ we denote by $\acx(U)$ the \emph{absolutely convex hull} of a subset $U\subset X$. 

\begin{rem}\label{rem:gen_mixed_top}
Let $(X,\tau)$ be a locally convex Hausdorff space, $\mathcal{B}$ the family of $\tau$-bounded sets and 
$\widetilde{\tau}$ a locally convex Hausdorff topology on $X$.
\begin{enumerate}
\item[(a)] The sets $\acx(\bigcup_{B\in\mathcal{B}}U_{B}\cap B)$ where each $U_{B}$ is a 
$0$-neighbourhood in $(X,\widetilde{\tau})$ for $B\in\mathcal{B}$ form a basis of absolutely convex $0$-neighbourhoods for 
$\widetilde{\gamma}$ since $\widetilde{\gamma}$ is the finest locally convex Hausdorff topology on $X$ 
which coincides with $\widetilde{\tau}$ on $\tau$-bounded sets. 
Further, we note that we may restrict to absolutely convex $\tau$-bounded sets $B$ and absolutely convex $\widetilde{\tau}$-closed 
$0$-neighbourhoods $U_{B}$ in $(X,\widetilde{\tau})$ as $(X,\tau)$ and $(X,\widetilde{\tau})$ are locally convex Hausdorff spaces. 
\item[(b)] If $\widetilde{\tau}_{0}$ is another locally convex Hausdorff topology on $X$ which coincides with $\widetilde{\tau}$ 
on $\tau$-bounded sets, then $\widetilde{\gamma}(\tau,\widetilde{\tau})=\widetilde{\gamma}(\tau,\widetilde{\tau}_{0})$.
\item[(c)] If $(X,\tau)$ is bornological and $(X,\widetilde{\gamma})$ satisfies $\operatorname{(B\tau B}\widetilde{\gamma})$, 
then $\widetilde{\tau}\leq\widetilde{\gamma}\leq\tau$. Indeed, $\widetilde{\tau}\leq\widetilde{\gamma}$ follows from the definition of $\widetilde{\gamma}$, and $\widetilde{\gamma}\leq\tau$ 
from \prettyref{prop:bornological_comp_top}.
\item[(d)] If $(X,\tau)$ is bornological, $\widetilde{\tau}=\tau$ and $(X,\widetilde{\gamma})$ satisfies 
$\operatorname{(B\tau B}\widetilde{\gamma})$, then $\widetilde{\gamma}=\widetilde{\gamma}(\tau,\tau)=\tau$. 
This follows directly from part (c).
\item[(e)] Let $(X,\widetilde{\gamma})$ satisfy $\operatorname{(B\tau B}\widetilde{\gamma})$. 
Then a subset of $X$ is $\widetilde{\gamma}$-compact (precompact, relatively compact) if and only if it is $\tau$-bounded and 
$\widetilde{\tau}$-compact (precompact, relatively compact). This follows directly from the definitions of $\widetilde{\gamma}$ 
and $\operatorname{(B\tau B}\widetilde{\gamma})$.
\end{enumerate}
\end{rem}

\prettyref{rem:gen_mixed_top} (e) is a generalisation of \cite[I.1.12 Proposition, p.~10]{cooper1978} 
and yields the following corollary, which itself generalises \cite[I.1.13 Proposition, p.~11]{cooper1978} 
by \prettyref{prop:mixed=gen_mixed} below.

\begin{cor}\label{cor:semi_montel_gen_mixed}
Let $(X,\tau)$ be a locally convex Hausdorff space, $\widetilde{\tau}$ a locally convex Hausdorff topology on $X$ and 
$(X,\widetilde{\gamma})$ satisfy $\operatorname{(B\tau B}\widetilde{\gamma})$. 
Then $(X,\widetilde{\gamma})$ is a semi-Montel space if and only if $(X,\tau)$ satisfies $\operatorname{(BBC})$ for 
$\widetilde{\tau}$.
\end{cor}

\begin{exa}\label{ex:gen_mixed_dual}
Let $(X,\tau)$ be a complete barrelled locally convex Hausdorff space. Then 
\[
\widetilde{\gamma}\coloneqq \widetilde{\gamma}(\beta(X',X),\sigma(X',X))=\tau_{\operatorname{c}}(X',X)
\]
and $(X',\widetilde{\gamma})$ satisfies $(\operatorname{B}\beta(X',X)\operatorname{B}\widetilde{\gamma})$ and is a semi-Montel space. 
Further, the evaluation map 
\[
\mathcal{J}_{X}\colon (X,\tau)\to (X',\widetilde{\gamma})_{b}',\;x\longmapsto[x'\mapsto x'(x)],
\]
is a topological isomorphism. 
\end{exa}
\begin{proof}
Since $(X,\tau)$ is barrelled, a subset of $X'=(X,\tau)'$ is $\beta(X',X)$-bounded if and only it is $\tau$-equicontinuous by 
\cite[Theorem 33.2, p.~349]{treves2006}. Now, the completeness of $(X,\tau)$ and \cite[\S 21.9, (7), p.~271]{koethe1969} 
imply that $\tau_{\operatorname{c}}(X',X)$ is the finest locally convex Hausdorff topology that coincides with $\sigma(X',X)$ 
on the $\tau$-equicontinuous subsets of $X'$. 
This yields $\widetilde{\gamma}(\beta(X',X),\sigma(X',X))=\tau_{\operatorname{c}}(X',X)$ by the definition of $\widetilde{\gamma}$. 
Again, due to \cite[Theorem 33.2, p.~349]{treves2006} a subset of $X'$ is $\beta(X',X)$-bounded if and only it is 
$\tau_{\operatorname{c}}(X',X)$-bounded since $\sigma(X',X)\leq\tau_{\operatorname{c}}(X',X)\leq\beta(X',X)$. 
Hence $(X',\widetilde{\gamma})$ satisfies $(\operatorname{B}\beta(X',X)\operatorname{B}\widetilde{\gamma})$. 
It also follows from \cite[Theorem 33.2, p.~349]{treves2006} that $(X',\beta(X',X))$ satisfies 
$\operatorname{(BBC})$ for $\sigma(X',X)$, implying that $(X',\widetilde{\gamma})$ is a semi-Montel space 
by \prettyref{cor:semi_montel_gen_mixed}.

Further, as a subset of $X'$ is $\beta(X',X)$-bounded if and only if 
it is $\tau_{\operatorname{c}}(X',X)$-bounded and $\widetilde{\gamma}=\tau_{\operatorname{c}}(X',X)$, 
the evaluation map $\mathcal{J}_{X}$ is a topological isomorphism by \cite[11.2.2 Proposition, p.~222]{jarchow1981} 
and the Mackey--Arens theorem. 
\end{proof}

In view of \prettyref{prop:mixed=gen_mixed} below, \prettyref{ex:gen_mixed_dual} improves 
\cite[I.2.A Examples, p.~20--21]{cooper1978} and \cite[Example E), p.~66]{wiweger1961} where $(X,\tau)$ is a Fr\'echet 
resp.~Banach space. \prettyref{ex:gen_mixed_dual} also shows that $(X',\tau_{\operatorname{c}}(X',X))$ is a predual for 
any complete barrelled locally convex Hausdorff space $(X,\tau)$. 
However, this predual may not have the properties of a predual one is looking for. 
For instance, if $(X,\tau)$ is completely normable, one is naturally looking for a completely normable predual as well. 
But the predual $(X',\tau_{\operatorname{c}}(X',X))$ is a semi-Montel space, so not normable unless $X$ is finite-dimensional. 

Next, we present a non-trivial sufficient condition that guarantees that $(X,\widetilde{\gamma})$ satisfies 
$\operatorname{(B\tau B}\widetilde{\gamma})$. 
Let us recall the definition of the mixed topology in the sense of \cite[p.~5--6]{cooper1978}. 
Let $(X,\tau)$ be a bornological locally convex Hausdorff space, $\mathcal{B}$ the family of $\tau$-bounded sets 
and $\widetilde{\tau}$ a locally convex Hausdorff topology on $X$ such that $\widetilde{\tau}\leq\tau$ 
(see \prettyref{prop:bornological_comp_top}). 
Suppose that $\mathcal{B}$ is of \emph{countable type}, i.e.~it has a countable basis, and that $(X,\tau)$ satisfies 
$\operatorname{(BBCl)}$ for $\widetilde{\tau}$. Then it follows that $\mathcal{B}$ has a countable basis $(B_{n})_{n\in\N}$ 
consisting of absolutely convex $\widetilde{\tau}$-closed sets such that $2B_{n}\subset B_{n+1}$ for all $n\in\N$ 
(see \cite[p.~6]{cooper1978}). For a sequence $(U_{n})_{n\in\N}$ of absolutely convex $0$-neighbourhoods in $(X,\widetilde{\tau})$ we set 
\[
\mathcal{U}((U_{n})_{n\in\N})\coloneqq\bigcup_{n=1}^{\infty}\sum_{k=1}^{n}(U_{k}\cap B_{k}).
\]
The family of such sets forms a basis of $0$-neighbourhoods of a locally convex Hausdorff topology on $X$, which is called 
the \emph{mixed topology} and denoted by $\gamma \coloneqq \gamma(\tau,\widetilde{\tau})$. 
We note that our definition of the mixed topology is slightly less general than the one given in \cite[p.~5--6]{cooper1978} 
since we only consider the case where the bornology $\mathcal{B}$ is induced by a topology, namely $\tau$, i.e.~we only consider 
the case of von Neumann bornologies. If $\tau$ is induced by a norm $\|\cdot\|$, then $\mathcal{B}$ is of countable type 
and the mixed topology defined above coincides with the mixed topology in the sense of \cite[p.~49]{wiweger1961}. 
Besides normable spaces, examples of spaces $(X,\tau)$ with $\mathcal{B}$ of countable type are df-spaces, in particular 
gDF-spaces and DF-spaces (see \cite[p.~257]{jarchow1981}). 

\begin{prop}[{\cite[I.1.5 Proposition (iii), I.1.11 Proposition, p.~7, 10]{cooper1978}}]\label{prop:mixed=gen_mixed}
Let $(X,\tau)$ be a bornological locally convex Hausdorff space, 
the family $\mathcal{B}$ of $\tau$-bounded sets be of countable type and $(X,\tau)$ satisfy 
$\operatorname{(BBCl)}$ for some $\widetilde{\tau}\leq\tau$. Then $\widetilde{\gamma}=\gamma$ and 
$(X,\widetilde{\gamma})$ satisfies $\operatorname{(B\tau B}\widetilde{\gamma})$.
\end{prop}

However, we note that the condition that $\mathcal{B}$ is of countable type is not a necessary condition for 
$(X,\widetilde{\gamma})$ satisfying $\operatorname{(B\tau B}\widetilde{\gamma})$ 
by \prettyref{ex:gen_mixed_dual} or by \cite[p.~272, 274, 276]{bierstedt1993} where in the latter case 
$X=\mathcal{HV}(\Omega)$ is a weighted space of holomorphic functions on a (balanced) open set $\Omega\subset\C^{N}$, $\tau=\tau_{\mathcal{V}}$ is the weighted topology 
w.r.t.~to a family $\mathcal{V}$ of non-negative upper semicontinuous functions on $\Omega$ such that 
$\widetilde{\tau}\coloneqq\tau_{\operatorname{co}}\leq\tau_{\mathcal{V}}$ and $\widetilde{\gamma}=\overline{\tau}$ (see \prettyref{ex:borno_frechet_hypo} as well with $N=1$, i.e.~$d=2$ there, and $P(\partial)$ being the Cauchy--Riemann operator).

Now, we show that $(X_{\mathcal{B},\widetilde{\tau}}',\beta)$ coincides with the strong dual 
of $(X,\widetilde{\gamma})$ under suitable assumptions. The proof is an adaptation of the proofs of 
\cite[I.1.7 Corollary, p.~8]{cooper1978}, \cite[I.1.18 Proposition, p.~15]{cooper1978} and \cite[I.1.20 Proposition, p.~16]{cooper1978}.

\begin{prop}\label{prop:predual_is_dual}
Let $(X,\tau)$ be a bornological locally convex Hausdorff space, $\mathcal{B}$ the family of $\tau$-bounded sets, $\widetilde{\tau}$ 
a locally convex Hausdorff topology on $X$ and $(X,\widetilde{\gamma})$ satisfy $\operatorname{(B\tau B}\widetilde{\gamma})$. 
Then the following assertions hold.
\begin{enumerate}
\item[(a)] $(X,\widetilde{\gamma})'$ is a closed subspace of $(X,\tau)_{b}'$, and 
$(X,\widetilde{\gamma})'=X_{\mathcal{B},\widetilde{\tau}}'$. Further, it holds
\[
 \beta((X,\widetilde{\gamma})',(X,\widetilde{\gamma}))
=\beta=\beta(X_{\mathcal{B},\widetilde{\tau}}',(X,\tau))
=\beta((X,\tau)',(X,\tau))_{\mid (X,\widetilde{\gamma})'}.
\]
In particular, $(X,\widetilde{\gamma})_{b}'$ is complete. 
\item[(b)] If $(X,\tau)$ satisfies $\operatorname{(BBCl)}$ for $\widetilde{\tau}$, then 
$(X,\widetilde{\gamma})'$ is the closure of $(X,\widetilde{\tau})'$ in $(X,\tau)_{b}'$.
\item[(c)] Let $(X,\tau)$ satisfy $\operatorname{(BBCl)}$ for $\widetilde{\tau}$. 
A subset of $X$ is $\sigma(X,(X,\widetilde{\gamma})')$-compact if and only if it is $\tau$-bounded and 
$\sigma(X,(X,\widetilde{\tau})')$-compact.
\end{enumerate}
\end{prop}
\begin{proof}
(a) We start with the proof of $(X,\widetilde{\gamma})'=X_{\mathcal{B},\widetilde{\tau}}'$. 
We have $(X,\widetilde{\gamma})'\subset X_{\mathcal{B},\widetilde{\tau}}'$ because $\widetilde{\gamma}$ coincides with 
$\widetilde{\tau}$ on $\tau$-bounded sets. Let us consider the other inclusion. 
Let $u\in X_{\mathcal{B},\widetilde{\tau}}'$ and $V$ be an absolutely convex $0$-neighbourhood in $\K$. 
For every $B\in\mathcal{B}$ we have
\[
u^{-1}(V)\cap B=u_{\mid B}^{-1}(V)=u_{\mid B}^{-1}(V)\cap B
\]
and $U_{B}\coloneqq u_{\mid B}^{-1}(V)$ is a $0$-neighbourhood for $\widetilde{\tau}$ since $u\in X_{\mathcal{B},\widetilde{\tau}}'$. 
Thus the set $\acx(\bigcup_{B\in\mathcal{B}}U_{B}\cap B)$ is an absolutely convex $0$-neighbourhood
for $\widetilde{\gamma}$ contained in $u^{-1}(V)$, implying $u\in (X,\widetilde{\gamma})'$. 

Due to \prettyref{rem:gen_mixed_top} (c) we have $\widetilde{\gamma}\leq\tau$. 
This implies that $(X,\widetilde{\gamma})'\subset(X,\tau)'$. 
Further, it is easily seen that the $\beta((X,\tau)',(X,\tau))$-limit 
$f\in(X,\tau)'$ of a net $(f_{\iota})_{\iota\in I}$ in $(X,\widetilde{\gamma})'$ belongs to 
$X_{\mathcal{B},\widetilde{\tau}}'$. Due to the first part of (a) we get that $f\in(X,\widetilde{\gamma})'$, which means that 
$(X,\widetilde{\gamma})'$ is closed in $(X,\tau)_{b}'$. The rest follows from the completeness of $(X,\tau)_{b}'$ and the assumption 
that $(X,\widetilde{\gamma})$ satisfies $\operatorname{(B\tau B}\widetilde{\gamma})$.

(b) Next, we show that $(X,\widetilde{\tau})'$ is dense in $(X,\widetilde{\gamma})'$. We know that 
$(X,\widetilde{\tau})'$ is a subspace of $(X,\widetilde{\gamma})'$ by \prettyref{rem:gen_mixed_top} (c). 
Due to $\operatorname{(BBCl)}$ for $\widetilde{\tau}$ the family $\mathcal{B}$ has a basis $\mathcal{B}_{0}$ of 
$\tau$-bounded sets which are absolutely convex and $\widetilde{\tau}$-closed (see \prettyref{rem:BBC_CNC} (a)). 
Let $u\in(X,\widetilde{\gamma})'$, $B\in\mathcal{B}_{0}$ and $\varepsilon>0$. 
Then there is an absolutely convex $\widetilde{\tau}$-closed $0$-neighbourhood 
$U$ in $(X,\widetilde{\tau})$ such that $|u(x)|\leq\varepsilon$ for all $x\in U\cap B$. 
This means that $u\in\varepsilon (U\cap B)^{\circ}$ 
where the polar set is taken w.r.t.~the dual pairing $\langle X,(X,\widetilde{\tau})'\rangle$. 
The set $U$ is absolutely convex and $\widetilde{\tau}$-closed and thus $\sigma(X,(X,\widetilde{\tau})')$-closed by 
\cite[8.2.5 Proposition, p.~149]{jarchow1981}. By the same reasoning the set $B$ is $\sigma(X,(X,\widetilde{\tau})')$-closed. 
Hence we get 
\[
(U\cap B)^{\circ}=\overline{\acx(U^{\circ}\cup B^{\circ})}\subset \overline{(U^{\circ}+ B^{\circ})}
\]
by \cite[8.2.4 Corollary, p.~149]{jarchow1981} where the closures are taken w.r.t.~$\sigma((X,\widetilde{\tau})',X)$. 
The polar set $U^{\circ}$ is $\sigma((X,\widetilde{\tau})',X)$-compact by the Alaoglu--Bourbaki theorem and the polar set 
$B^{\circ}$ is $\sigma((X,\widetilde{\tau})',X)$-closed by \cite[8.2.1 Proposition (a), p.~148]{jarchow1981}. 
Therefore the sum $U^{\circ}+ B^{\circ}$ is $\sigma((X,\widetilde{\tau})',X)$-closed and so  
$(U\cap B)^{\circ}\subset (U^{\circ}+ B^{\circ})$. We deduce that $u\in\varepsilon (U^{\circ}+ B^{\circ})$, which yields that 
there is $v\in\varepsilon U^{\circ}\subset (X,\widetilde{\tau})'$ such that $u-v\in \varepsilon B^{\circ}$, 
i.e.~$|u(x)-v(x)|\leq\varepsilon$ for all $x\in B$.

(c) $\Rightarrow$ This implication follows from $\sigma(X,(X,\widetilde{\tau})')\leq\sigma(X,(X,\widetilde{\gamma})')$, 
the Mackey theorem and that a $\widetilde{\gamma}$-bounded set is $\tau$-bounded because of 
$\operatorname{(B\tau B}\widetilde{\gamma})$. 

$\Leftarrow$ Let $B\subset X$ be $\tau$-bounded. 
Then $B$ is $\widetilde{\gamma}$-bounded because of $\operatorname{(B\tau B}\widetilde{\gamma})$, 
and $\mathcal{J}(B)$ an equicontinuous subset of $((X,\widetilde{\gamma})_{b}')'$ 
by \prettyref{prop:bounded_equicont}. The topologies $\sigma((X,\widetilde{\gamma})_{b}')',(X,\widetilde{\gamma})')$ and 
$\sigma((X,\widetilde{\gamma})_{b}')',(X,\widetilde{\tau})')$ coincide on the equicontinuous set $\mathcal{J}(B)$ by 
\cite[Satz 1.4, p.~16]{kaballo2014} since $(X,\widetilde{\tau})'$ is dense in $(X,\widetilde{\gamma})_{b}'$ by part (b). 
Hence $\sigma(X,(X,\widetilde{\gamma})')$ and $\sigma(X,(X,\widetilde{\tau})')$ coincide on $B$. 
Thus, if $B$ is in addition $\sigma(X,(X,\widetilde{\tau})')$-compact, we get that $B$ is also 
$\sigma(X,(X,\widetilde{\gamma})')$-compact.
\end{proof}

\prettyref{prop:predual_is_dual} (c) in combination with \cite[23.18 Proposition, p.~270]{meisevogt1997} 
and \cite[8.2.5 Proposition, p.~149]{jarchow1981} directly implies the following statement, 
which is a generalisation of \cite[I.1.21 Corollary, p.~16]{cooper1978}.

\begin{cor}\label{cor:semi_reflexive_gen_mixed}
Let $(X,\tau)$ be a bornological locally convex Hausdorff space, $\mathcal{B}$ the family of $\tau$-bounded sets, $\widetilde{\tau}$ 
a locally convex Hausdorff topology on $X$ and $(X,\widetilde{\gamma})$ satisfy $\operatorname{(B\tau B}\widetilde{\gamma})$. 
Then $(X,\widetilde{\gamma})$ is semi-reflexive and $(X,\tau)$ satisfies $\operatorname{(BBCl)}$ for $\widetilde{\tau}$ 
if and only if $\mathcal{B}$ has a basis of absolutely convex $\sigma(X,(X,\widetilde{\tau})')$-compact sets.
\end{cor}

\begin{cor}\label{cor:semi_reflexive_BBCl_CNC}
Let $(X,\tau)$ be a bornological locally convex Hausdorff space, $\widetilde{\tau}$ a locally convex Hausdorff topology on $X$ 
and $(X,\widetilde{\gamma})$ satisfy $\operatorname{(B\tau B}\widetilde{\gamma})$ with 
$\widetilde{\gamma}\coloneqq\widetilde{\gamma}(\tau,\widetilde{\tau})$. 
Then the following assertions are equivalent.
\begin{enumerate}
\item[(a)] $(X,\tau)$ satisfies $\operatorname{(BBCl)}$ and $\operatorname{(CNC)}$ 
for $\widetilde{\tau}$ and $(X,\widetilde{\gamma})$ is semi-reflexive.
\item[(b)] $(X,\tau)$ satisfies $\operatorname{(BBC)}$ and $\operatorname{(CNC)}$ for $\sigma(X,(X,\widetilde{\tau})')$.
\end{enumerate}
If one, thus both, of the preceding assertions holds, then we have 
\[
 \widetilde{\gamma}(\tau,\sigma(X,(X,\widetilde{\tau})'))
=\widetilde{\gamma}(\tau,\sigma(X,(X,\widetilde{\gamma})'))
=\tau_{\operatorname{c}}(X,(X,\widetilde{\gamma})_{b}')
\]
and $(X,\widetilde{\gamma}(\tau,\sigma(X,(X,\widetilde{\tau})')))$ satisfies 
$\operatorname{(B\tau B}\widetilde{\gamma}(\tau,\sigma(X,(X,\widetilde{\tau})')))$ as well as 
\[
 (X,\widetilde{\gamma})_{b}'
=(X,\widetilde{\gamma}(\tau,\sigma(X,(X,\widetilde{\tau})')))_{b}'.
\]
\end{cor}
\begin{proof}
First, we show that $(X,\tau)$ satisfies $\operatorname{(CNC)}$ for $\widetilde{\tau}$ if and only if 
$(X,\tau)$ satisfies $\operatorname{(CNC)}$ for $\sigma(X,(X,\widetilde{\tau})')$. 
Let $(X,\tau)$ satisfy $\operatorname{(CNC)}$ for $\widetilde{\tau}$. Then $(X,\tau)$ satisfies $\operatorname{(CNC)}$ for 
$\sigma(X,(X,\widetilde{\tau})')$ by \prettyref{rem:BBC_CNC} (f) with $\widetilde{\tau}_{0}\coloneqq \sigma(X,(X,\widetilde{\tau})')$.
On the other hand, let $(X,\tau)$ satisfy $\operatorname{(CNC)}$ for $\sigma(X,(X,\widetilde{\tau})')$. 
Then $\tau$ has a $0$-neighbourhood basis $\mathcal{U}_{0}$ of absolutely convex $\sigma(X,(X,\widetilde{\tau})')$-closed sets. 
Since $\sigma(X,(X,\widetilde{\tau})')\leq \widetilde{\tau}$, the elements of $\mathcal{U}_{0}$ are also 
$\widetilde{\tau}$-closed. Thus $(X,\tau)$ satisfies $\operatorname{(CNC)}$ for $\widetilde{\tau}$. 

Now, the equivalence (a)$\Leftrightarrow$(b) follows from the observation above and \prettyref{cor:semi_reflexive_gen_mixed}. 
Let one of the assertions (a) or (b), thus both, hold. By the proof of \prettyref{prop:predual_is_dual} (c) we know that 
$\sigma(X,(X,\widetilde{\tau})')$ and $\sigma(X,(X,\widetilde{\gamma})')$ coincide on $\tau$-bounded sets. 
Due to \prettyref{rem:gen_mixed_top} (b) we obtain that 
\[
\widetilde{\gamma}(\tau,\sigma(X,(X,\widetilde{\tau})'))
=\widetilde{\gamma}(\tau,\sigma(X,(X,\widetilde{\gamma})')).
\]
Since $(X,\widetilde{\gamma})$ is semi-reflexive, all equicontinuous subsets of $((X,\widetilde{\gamma})_{b}')'$ are of the form 
$\mathcal{J}(B)$ for some $\widetilde{\gamma}$-bounded set $B\subset X$, and all 
$\widetilde{\gamma}$-bounded subsets of $X$ are of the form $\mathcal{J}^{-1}(A)$ for some equicontinuous set 
$A\subset((X,\widetilde{\gamma})_{b}')'$ by \prettyref{prop:bounded_equicont}. 
Now, the completeness of $(X,\widetilde{\gamma})_{b}'$ by \prettyref{prop:predual_is_dual} (a) and 
\cite[\S 21.9, (7), p.~271]{koethe1969} imply that $\tau_{\operatorname{c}}(\mathcal{J}(X),(X,\widetilde{\gamma})_{b}')$ 
is the finest locally convex Hausdorff topology that coincides with $\sigma(\mathcal{J}(X),(X,\widetilde{\gamma})')$ 
on the equicontinuous subsets of $\mathcal{J}(X)=((X,\widetilde{\gamma})_{b}')'$. 
Due to $\operatorname{(B\tau B}\widetilde{\gamma})$ this yields
\[
 \widetilde{\gamma}(\tau,\sigma(X,(X,\widetilde{\gamma})'))
=\tau_{\operatorname{c}}(X,(X,\widetilde{\gamma})_{b}').
\]
Further, the Mackey theorem in combination with $(X,\widetilde{\gamma})$ satisfying $\operatorname{(B\tau B}\widetilde{\gamma})$ 
yields that a subset of $X$ is $\widetilde{\gamma}(\tau,\sigma(X,(X,\widetilde{\tau})'))$-bounded if and only 
if it is $\widetilde{\gamma}$-bounded if and only if it is $\tau$-bounded. 
This implies that the space $(X,\widetilde{\gamma}(\tau,\sigma(X,(X,\widetilde{\tau})')))$ satisfies 
$\operatorname{(B\tau B}\widetilde{\gamma}(\tau,\sigma(X,(X,\widetilde{\tau})')))$. 
The rest of the statement follows from the Mackey--Arens theorem.
\end{proof}

Combining \prettyref{thm:general_dixmier_ng}, \prettyref{cor:semi_reflexive_BBCl_CNC} and \prettyref{prop:predual_is_dual} (a), 
we get the biduality we were aiming for.

\begin{cor}\label{cor:general_dixmier_ng_semi_reflexive}
Let $(X,\tau)$ be a bornological locally convex Hausdorff space satisfying $\operatorname{(BBCl)}$ and $\operatorname{(CNC)}$ for some 
$\widetilde{\tau}$ and $(X,\widetilde{\gamma})$ a semi-reflexive space satisfying $\operatorname{(B\tau B}\widetilde{\gamma})$. 
Then $(X,\widetilde{\gamma})_{b}'$ is a complete barrelled locally convex Hausdorff space and the evaluation map 
\[
\mathcal{I}\colon(X,\tau)\to ((X,\widetilde{\gamma})_{b}')_{b}',\; x\longmapsto [x'\mapsto x'(x)], 
\]
is a topological isomorphism. In particular, $(X,\tau)$ is a bidual space with semi-reflexive prebidual 
$((X,\widetilde{\gamma}),\mathcal{I})$. 
\end{cor}

Next, we show that $\operatorname{(BBC)}$ and $\operatorname{(CNC)}$ 
in \prettyref{thm:general_dixmier_ng} are also necessary conditions for the existence of a 
complete barrelled predual. 
Let $X$ be a dual space with predual $(Y,\varphi)$. We define two locally convex Hausdorff topology on $X$ w.r.t.~the dual paring 
$\langle X,Y,\varphi\rangle$. 
We define the systems of seminorms 
\[
p_{N}(x)\coloneqq\sup_{y\in N}|\varphi(x)(y)|,\quad x\in X,
\]
for finite resp.~compact sets $N\subset Y$, which induce two locally convex Hausdorff topologies on $X$ w.r.t.~the dual paring 
$\langle X,Y,\varphi\rangle$ and we denote these topologies by $\sigma_{\varphi}(X,Y)$ 
resp.~$\tau_{\operatorname{c},\varphi}(X,Y)$. 
We recall from \prettyref{ex:gen_mixed_dual} that the evaluation map 
$\mathcal{J}_{Y}\colon Y\to (Y',\tau_{\operatorname{c}}(Y',Y))_{b}'$, $y\longmapsto[y'\mapsto y'(y)]$, is a topological 
isomorphism if $Y$ is complete and barrelled.

\begin{prop}\label{prop:predual_to_BBC_CNC}
Let $(X,\tau)$ be a dual space with complete barrelled predual $(Y,\varphi)$ 
and $\mathcal{B}$ the family of $\tau$-bounded sets. Then the following assertions hold.
\begin{enumerate}
\item[(a)] $(X,\tau)$ satisfies $\operatorname{(BBC)}$ and $\operatorname{(CNC)}$ for $\sigma_{\varphi}(X,Y)$, 
\[
\widetilde{\gamma}_{\varphi}\coloneqq\widetilde{\gamma}(\tau,\sigma_{\varphi}(X,Y))=\tau_{\operatorname{c},\varphi}(X,Y)
\]
and $(X,\widetilde{\gamma}_{\varphi})$ satisfies $(\operatorname{B}\tau \operatorname{B}\widetilde{\gamma}_{\varphi})$. 
\item[(b)] The three maps 
\[
\kappa_{\varphi}\colon (Y',\tau_{\operatorname{c}}(Y',Y))_{b}'\to (X,\widetilde{\gamma}_{\varphi})_{b}',\;
y''\mapsto y''\circ\varphi,
\]
and $\kappa_{\varphi}\circ \mathcal{J}_{Y}\colon Y\to (X,\widetilde{\gamma}_{\varphi})_{b}'$ as well as
\[
\mathcal{I}_{\varphi}\colon(X,\tau)\to ((X,\widetilde{\gamma}_{\varphi})_{b}')_{b}',\; x\longmapsto [x'\mapsto x'(x)], 
\]
are topological isomorphisms with $(\kappa_{\varphi}\circ\mathcal{J}_{Y})^{t}=\varphi\circ\mathcal{I}_{\varphi}^{-1}$ and
\[
 (X,\widetilde{\gamma}_{\varphi})_{b}'
=(X_{\mathcal{B},\sigma_{\varphi}(X,Y)}',\beta(X_{\mathcal{B},\sigma_{\varphi}(X,Y)}',(X,\tau))).
\]
Here, $(\kappa_{\varphi}\circ\mathcal{J}_{Y})^{t}$ denotes the dual map of $\kappa_{\varphi}\circ\mathcal{J}_{Y}$.
\item[(c)] Let $x'\colon X\to\K$. Then $x'\in (X,\widetilde{\gamma}_{\varphi})'$ if and only if $x'\in (X,\sigma_{\varphi}(X,Y))'$ 
if and only if there is a (unique) $y\in Y$ such that $x'=\varphi(\cdot)(y)$.
\end{enumerate}
\end{prop} 
\begin{proof}
(a) The space $Y_{b}'$ satisfies $\operatorname{(BBC)}$ and $\operatorname{(CNC)}$ for $\sigma(Y',Y)$ 
by \cite[p.~116]{bierstedt1992}, which directly yields that $(X,\tau)$ satisfies $\operatorname{(BBC)}$ and $\operatorname{(CNC)}$ 
for $\sigma_{\varphi}(X,Y)$. Due to \prettyref{ex:gen_mixed_dual} the evaluation map $\mathcal{J}_{Y}$ is a topological 
isomorphism and we have 
\[
\widetilde{\gamma}\coloneqq\widetilde{\gamma}(\beta(Y',Y),\sigma(Y',Y))=\tau_{\operatorname{c}}(Y',Y)
\]
as well as that $(Y',\widetilde{\gamma})$ satisfies $(\operatorname{B}\beta(Y',Y)\operatorname{B}\widetilde{\gamma})$, which implies 
\[
\widetilde{\gamma}_{\varphi}=\widetilde{\gamma}(\tau,\sigma_{\varphi}(X,Y))=\tau_{\operatorname{c},\varphi}(X,Y)
\]
and that $(X,\widetilde{\gamma}_{\varphi})$ satisfies $(\operatorname{B}\tau \operatorname{B}\widetilde{\gamma}_{\varphi})$. 

(b) The map $\kappa_{\varphi}$ is a topological isomorphism since $\varphi$ is a topological isomorphism and 
$\widetilde{\gamma}_{\varphi}=\tau_{\operatorname{c},\varphi}(X,Y)$ by part (a). Thus $\kappa_{\varphi}\circ \mathcal{J}_{Y}$ is a 
topological isomorphism as well. It follows that the map $\psi\circ\varphi \colon(X,\tau)\to ((X,\widetilde{\gamma}_{\varphi})_{b}')_{b}'$ with
\[
\psi\colon Y_{b}'\to ((X,\widetilde{\gamma}_{\varphi})_{b}')_{b}',\;
\psi(y')\coloneqq y'\circ (\kappa_{\varphi}\circ\mathcal{J}_{Y})^{-1},
\]
is a topological isomorphism. Let $x'\in(X,\widetilde{\gamma}_{\varphi})'$. Then we have 
$(\kappa_{\varphi}\circ\mathcal{J}_{Y})^{-1}(x')=\mathcal{J}_{Y}^{-1}(x'\circ\varphi^{-1})$ and 
\begin{align*}
 (\psi\circ\varphi)(x)(x')
&=\varphi(x)(\mathcal{J}_{Y}^{-1}(x'\circ\varphi^{-1}))
 =\mathcal{J}_{Y}(\mathcal{J}_{Y}^{-1}(x'\circ\varphi^{-1}))(\varphi(x))\\
&=(x'\circ\varphi^{-1})(\varphi(x))
 =x'(x)
 =\mathcal{I}_{\varphi}(x)(x')
\end{align*}
for all $x\in X$, proving that the map $\mathcal{I}_{\varphi}$ is a topological isomorphism.  
Looking at the proof of \prettyref{prop:predual_is_dual} (a), we see that  
\[
 (X,\widetilde{\gamma}_{\varphi})_{b}'
=(X_{\mathcal{B},\sigma_{\varphi}(X,Y)}',\beta(X_{\mathcal{B},\sigma_{\varphi}(X,Y)}',(X,\tau))).
\]
holds due to $(X,\widetilde{\gamma}_{\varphi})$ satisfying $(\operatorname{B}\tau \operatorname{B}\widetilde{\gamma}_{\varphi})$ 
even without the assumption that $(X,\tau)$ is bornological. 

Next, we show that $(\kappa_{\varphi}\circ\mathcal{J}_{Y})^{t}=\varphi\circ\mathcal{I}_{\varphi}^{-1}$. We have
\begin{equation}\label{eq:predual_to_BBC_CNC}
 (\kappa_{\varphi}\circ \mathcal{J}_{Y})(y)(x)
=(\mathcal{J}_{Y}(y)\circ\varphi)(x)
=\mathcal{J}_{Y}(y)(\varphi(x))
=\varphi(x)(y)
\end{equation}
for all $y\in Y$ and $x\in X$. Let $x''\in ((X,\widetilde{\gamma}_{\varphi})_{b}')'$. Then there is a unique $x\in X$ such that 
$x''=\mathcal{I}_{\varphi}(x)$ and 
\begin{align*}
 (\kappa_{\varphi}\circ\mathcal{J}_{Y})^{t}(x'')(y)
&=x''((\kappa_{\varphi}\circ\mathcal{J}_{Y})(y))
 \underset{\eqref{eq:predual_to_BBC_CNC}}{=}x''(\varphi(\cdot)(y))
 =\mathcal{I}_{\varphi}(x)(\varphi(\cdot)(y))\\
&=\varphi(x)(y)
 =\varphi(\mathcal{I}_{\varphi}^{-1}(x''))(y)
\end{align*}
for all $y\in Y$.

(c) This part follows from part (b), \eqref{eq:predual_to_BBC_CNC} and \cite[Chap.~IV, \S1, 1.2, p.~124]{schaefer1971}. 
\end{proof}

Due to \prettyref{thm:general_dixmier_ng} and \prettyref{prop:predual_to_BBC_CNC} (a) we have that 
$\operatorname{(BBC)}$ and $\operatorname{(CNC)}$ are necessary and sufficient conditions for the 
existence of a complete barrelled predual of a bornological space.

\begin{cor}\label{cor:existence_cb_predual}
Let $(X,\tau)$ be a bornological locally convex Hausdorff space. 
Then the following assertions are equivalent.
\begin{enumerate}
\item[(a)] $(X,\tau)$ has a complete barrelled predual.
\item[(b)] $(X,\tau)$ has a semi-Montel prebidual $(Y,\varphi)$ such that $Y_{b}'$ is complete.
\item[(c)] $(X,\tau)$ satisfies $\operatorname{(BBC)}$ and $\operatorname{(CNC)}$ 
for some $\widetilde{\tau}$ and $(X,\widetilde{\gamma})$ satisfies $\operatorname{(B\tau B}\widetilde{\gamma})$.
\item[(d)] $(X,\tau)$ satisfies $\operatorname{(BBCl)}$ and $\operatorname{(CNC)}$ 
for some $\widetilde{\tau}$ and $(X,\widetilde{\gamma})$ is semi-reflexive and satisfies $\operatorname{(B\tau B}\widetilde{\gamma})$.
\item[(e)] $(X,\tau)$ satisfies $\operatorname{(BBC)}$ and $\operatorname{(CNC)}$ 
for some $\widetilde{\tau}$.
\end{enumerate}
\end{cor}
\begin{proof}
(b)$\Rightarrow$(a) This implication follows from the observations that semi-Montel space are semi-reflexive and 
semi-reflexive locally convex Hausdorff spaces are distinguished by \cite[11.4.1 Proposition, p.~227]{jarchow1981}. Hence the tuple $(Y_{b}',\varphi)$ is a complete barrelled predual of $(X,\tau)$.

(a)$\Rightarrow$(c) This implication follows from \prettyref{prop:predual_to_BBC_CNC} (a) with 
$\widetilde{\tau}\coloneqq\sigma_{\varphi}(X,Y)$ for a complete barrelled predual $(Y,\varphi)$ of $(X,\tau)$. 

(c)$\Rightarrow$(b) Due to \prettyref{cor:semi_montel_gen_mixed} $(X,\widetilde{\gamma})$ is a semi-Montel space and 
thus semi-reflexive. Hence the implication follows from \prettyref{cor:general_dixmier_ng_semi_reflexive} 
with $Y\coloneqq (X,\widetilde{\gamma})$ and $\varphi\coloneqq\mathcal{I}$. 

(c)$\Rightarrow$(d) This implication follows from \prettyref{cor:semi_montel_gen_mixed} 
and the observation that semi-Montel space are semi-reflexive. 

(d)$\Rightarrow$(a) This implication follows from \prettyref{cor:general_dixmier_ng_semi_reflexive} with 
$Y\coloneqq (X,\widetilde{\gamma})$ and $\varphi\coloneqq\mathcal{I}$. 

(c)$\Rightarrow$(e) This implication is obvious. 

(e)$\Rightarrow$(a) This implication follows from \prettyref{thm:general_dixmier_ng}. 
\end{proof}

\begin{cor}\label{cor:existence_cbDF_predual}
Let $(X,\tau)$ be a bornological locally convex Hausdorff space. 
Then the following assertions are equivalent.
\begin{enumerate}
\item[(a)] $(X,\tau)$ has a complete barrelled DF-predual.
\item[(b)] $(X,\tau)$ has a semi-Montel prebidual $(Y,\varphi)$ such that $Y_{b}'$ is a complete DF-space.
\item[(c)] $(X,\tau)$ is a Fr\'echet space satisfying $\operatorname{(BBC)}$ and $\operatorname{(CNC)}$ 
for some $\widetilde{\tau}$ and $(X,\widetilde{\gamma})$ satisfies $\operatorname{(B\tau B}\widetilde{\gamma})$.
\item[(d)] $(X,\tau)$ is a Fr\'echet space satisfying $\operatorname{(BBCl)}$ and $\operatorname{(CNC)}$ 
for some $\widetilde{\tau}$ and $(X,\widetilde{\gamma})$ is semi-reflexive and satisfies $\operatorname{(B\tau B}\widetilde{\gamma})$.
\item[(e)] $(X,\tau)$  is a Fr\'echet space satisfying $\operatorname{(BBC)}$ and $\operatorname{(CNC)}$ 
for some $\widetilde{\tau}$.
\end{enumerate}
\end{cor}
\begin{proof}
(b)$\Rightarrow$(a), (c)$\Rightarrow$(d), (c)$\Rightarrow$(e) These implications follow from \prettyref{cor:existence_cb_predual}.

(a)$\Rightarrow$(c) This implication follows from the proof of the implication (a)$\Rightarrow$(c) 
of \prettyref{cor:existence_cb_predual} and the observation that the strong dual of a DF-space is a Fr\'echet space by 
\cite[12.4.2 Theorem, p.~258]{jarchow1981}. 

(c)$\Rightarrow$(b), (d)$\Rightarrow$(a), (e)$\Rightarrow$(a) These implications follow from the proof of the corresponding implications 
of \prettyref{cor:existence_cb_predual} and the observation that $F\coloneqq (X_{\mathcal{B},\widetilde{\tau}}',\beta)$ 
is a complete barrelled DF-space by the proof of \cite[5.~Corollary (b), p.~117--118]{bierstedt1992}, which coincides with 
$(X,\widetilde{\gamma})_{b}'$ by \prettyref{prop:predual_is_dual} (a) 
if $(X,\widetilde{\gamma})$ satisfies 
$\operatorname{(B\tau B}\widetilde{\gamma})$.
\end{proof}

\begin{cor}\label{cor:existence_frechet_predual}
Let $(X,\tau)$ be a bornological locally convex Hausdorff space and $\mathcal{B}$ the family of $\tau$-bounded sets. 
Then the following assertions are equivalent.
\begin{enumerate}
\item[(a)] $(X,\tau)$ has a Fr\'echet predual.
\item[(b)] $(X,\tau)$ has a semi-Montel prebidual $(Y,\varphi)$ such that $Y_{b}'$ is a Fr\'echet space.
\item[(c)] $(X,\tau)$ is a complete DF-space satisfying $\operatorname{(BBC)}$ and $\operatorname{(CNC)}$ 
for some $\widetilde{\tau}$ and $(X,\gamma)$ satisfies $\operatorname{(B\tau B}\gamma)$.
\item[(d)] $(X,\tau)$ is a complete DF-space satisfying $\operatorname{(BBCl)}$ and $\operatorname{(CNC)}$ for some 
$\widetilde{\tau}\leq\tau$ and $(X,\gamma)$ is semi-reflexive. 
\item[(e)] $(X,\tau)$ is a complete DF-space satisfying $\operatorname{(BBC)}$ and $\operatorname{(CNC)}$ 
for some $\widetilde{\tau}$. 
\item[(f)] $(X,\tau)$ is a complete DF-space satisfying $\operatorname{(BBC)}$ for some $\widetilde{\tau}$ and 
the Fr\'echet space $(X_{\mathcal{B},\widetilde{\tau}}',\beta)$ is distinguished. 
\end{enumerate}
\end{cor}
\begin{proof}
(c)$\Rightarrow$(e) This implication is obvious.

(b)$\Rightarrow$(a), (c)$\Rightarrow$(d) These implications follow from 
the proof of the corresponding implications of \prettyref{cor:existence_cb_predual} by noting that $\gamma=\widetilde{\gamma}$ and 
that $(X,\gamma)$ satisfies $\operatorname{(B\tau B}\gamma)$ by \prettyref{prop:mixed=gen_mixed} and \prettyref{rem:BBC_CNC} (c).

(a)$\Rightarrow$(c) This implication follows from the proof of the implication (a)$\Rightarrow$(c) 
of \prettyref{cor:existence_cb_predual}, the observation that the strong dual of a Fr\'echet predual $(Y,\varphi)$ of $(X,\tau)$ 
is a DF-space by \cite[12.4.5 Theorem, p.~260]{jarchow1981} and by noting that 
$\widetilde{\gamma}_{\varphi}=\gamma(\tau,\sigma_{\varphi}(X,Y))$ by \prettyref{prop:mixed=gen_mixed} and \prettyref{rem:BBC_CNC} (c).

(c)$\Rightarrow$(b), (d)$\Rightarrow$(a), (e)$\Rightarrow$(a) These implications follow from the proof of the corresponding implications 
of \prettyref{cor:existence_cb_predual} and the observation that $(X_{\mathcal{B},\widetilde{\tau}}',\beta)$ is a Fr\'echet space 
by \prettyref{prop:predual_complete} (b), which coincides with 
$(X,\gamma)_{b}'$ by \prettyref{prop:predual_is_dual} (a) 
since $\gamma=\widetilde{\gamma}$ and $(X,\gamma)$ satisfies 
$\operatorname{(B\tau B}\gamma)$ by \prettyref{prop:mixed=gen_mixed} and \prettyref{rem:BBC_CNC} (c).

(c)$\Rightarrow$(f) Due to the equivalence (a)$\Leftrightarrow$(e) the Fr\'echet space $(X_{\mathcal{B},\widetilde{\tau}}',\beta)$ 
is distinguished by \cite[Theorem 8.3.44, p.~261]{bonet1987} since 
$(X_{\mathcal{B},\widetilde{\tau}}',\beta)_{b}'$ is topologically isomorphic to the bornological space $(X,\tau)$.

(f)$\Rightarrow$(a) By \cite[1.~Theorem (Mujica), p.~115]{bierstedt1992} the map 
\[
\widetilde{\mathcal{I}}\colon(X,\tau)\to (X_{\mathcal{B},\widetilde{\tau}}',\beta)_{i}',\; x\longmapsto [x'\mapsto x'(x)], 
\]
is a topological isomorphism where $(X_{\mathcal{B},\widetilde{\tau}}',\beta)_{i}'$ 
is the inductive dual. Since $E\coloneqq(X_{\mathcal{B},\widetilde{\tau}}',\beta)$ 
is a Fr\'echet space, thus barrelled, every null sequence in $E_{b}'$ is equicontinuous by 
\cite[11.1.1 Proposition, p.~220]{jarchow1981}. Thus it follows from \cite[Observation 8.3.40 (b), p.~260]{bonet1987} that 
$E_{i}'=E_{b}'$ because $E_{b}'$ is bornological by \cite[Theorem 8.3.44, p.~261]{bonet1987} as 
$E$ is a distinguished Fr\'echet space. 
\end{proof} 

We note that the space $(X_{\mathcal{B},\widetilde{\tau}}',\beta)$ is distinguished 
by \cite[Proposition 8.3.45 (iii), p.~262]{bonet1987} if it is a quasi-normable Fr\'echet space. Sufficient conditions 
that $(X_{\mathcal{B},\widetilde{\tau}}',\beta)$ is a quasi-normable Fr\'echet space 
are given in \cite[4.~Remark, p.~117]{bierstedt1992}. The following example is a slight generalisation of 
\cite[3.~Examples B, p.~125--126]{bierstedt1992} where holomorphic functions are considered.

\begin{exa}\label{ex:borno_DF_hypo}
Let $\Omega\subset\R^{d}$ be open, $P(\partial)$ a hypoelliptic linear partial differential operator on $\mathcal{C}^{\infty}(\Omega)$ 
and $\mathcal{V}\coloneqq (v_{n})_{n\in\N}$ a \emph{drecreasing} family, i.e.~$v_{n+1}\leq v_{n}$ for all $n\in\N$, 
of continuous functions $v_{n}\colon\Omega\to(0,\infty)$. We define the inductive limit
\[
\mathcal{VC}_{P}(\Omega)\coloneqq\lim\limits_{\substack{\longrightarrow\\n\in \N}}\,\mathcal{C}_{P}v_{n}(\Omega)
\] 
of the Banach spaces $(\mathcal{C}_{P}v_{n}(\Omega),\|\cdot\|_{v_{n}})$, and equip 
$\mathcal{VC}_{P}(\Omega)$ with its locally convex inductive limit topology ${_{\mathcal{V}}\tau}$.
The space $(\mathcal{VC}_{P}(\Omega),{_{\mathcal{V}}\tau})$ is a strict inductive limit and thus ${_{\mathcal{V}}\tau}$ is Hausdorff by 
\cite[4.6.1 Theorem, p.~84]{jarchow1981} and $(\mathcal{VC}_{P}(\Omega),{_{\mathcal{V}}\tau})$ complete by 
\cite[4.6.4 Theorem, p.~86]{jarchow1981}. Further, the inductive limit is a (ultra)bornological DF-space by 
\cite[Proposition 25.16, p.~301]{meisevogt1997} and 
$\tau_{\operatorname{co}}\coloneqq{\tau_{\operatorname{co}}}_{\mid \mathcal{VC}_{P}(\Omega)}\leq {_{\mathcal{V}}\tau}$. 
The $\|\cdot\|_{v_{n}}$-closed unit balls $B_{\|\cdot\|_{v_{n}}}$ of $\mathcal{C}_{P}v_{n}(\Omega)$ 
are $\tau_{\operatorname{co}}$-closed. It follows that they are also $\tau_{\operatorname{co}}$-compact for all $n\in\N$ 
as $(\mathcal{C}_{P}(\Omega),\tau_{\operatorname{co}})$ is a Fr\'echet--Schwartz space. 
Since $\tau_{\operatorname{co}}\leq {_{\mathcal{V}}\tau}$, they are ${_{\mathcal{V}}\tau}$-closed as well and thus they 
form a basis of absolutely convex ${_{\mathcal{V}}\tau}$-bounded sets by \cite[Proposition 25.16, p.~301]{meisevogt1997}.
Therefore $(\mathcal{VC}_{P}(\Omega),{_{\mathcal{V}}\tau})$ satisfies $\operatorname{(BBC)}$ for $\tau_{\operatorname{co}}$. 
If $\mathcal{V}$ is \emph{regularly decreasing}, i.e.~for every $n\in\N$ there is $m\geq n$ such that for every $U\subset\Omega$ with 
$\inf_{x\in U}v_{m}(x)/v_{n}(x)>0$ we also have $\inf_{x\in U}v_{k}(x)/v_{n}(x)>0$ for all $k\geq m+1$, then the Fr\'echet space 
$(\mathcal{VC}_{P}(\Omega)_{\mathcal{B},\tau_{\operatorname{co}}}',\beta)$ is quasi-normable by \cite[p.~125--126]{bierstedt1992} 
and so distinguished. 
\end{exa}

\begin{cor}\label{cor:existence_banach_predual}
Let $(X,\tau)$ be a bornological locally convex Hausdorff space. 
Then the following assertions are equivalent.
\begin{enumerate}
\item[(a)] $(X,\tau)$ has a Banach predual.
\item[(b)] $(X,\tau)$ has a semi-Montel prebidual $(Y,\varphi)$ such that $Y_{b}'$ is completely norm\-able.
\item[(c)] $(X,\tau)$ is a completely normable space satisfying $\operatorname{(BBC)}$ and $\operatorname{(CNC)}$ 
for some $\widetilde{\tau}$ and $(X,\gamma)$ satisfies $\operatorname{(B\tau B}\gamma)$.
\item[(d)] $(X,\tau)$ is a completely normable space satisfying $\operatorname{(BBCl)}$ 
for some $\widetilde{\tau}\leq\tau$ and $(X,\gamma)$ is semi-reflexive. 
\item[(e)] $(X,\tau)$ is a completely normable space satisfying $\operatorname{(BBC)}$ for some $\widetilde{\tau}$.
\item[(f)] There is a norm $\vertiii{\cdot}$ on $X$ which induces $\tau$ such that $B_{\vertiii{\cdot}}$ $\widetilde{\tau}$-compact for some $\widetilde{\tau}$.
\end{enumerate}
\end{cor}
\begin{proof}
(b)$\Rightarrow$(a), (c)$\Rightarrow$(d), (c)$\Rightarrow$(e) These implications follow from the proof of the corresponding implications of \prettyref{cor:existence_frechet_predual}.

(a)$\Rightarrow$(c) This implication follows from the proof of the implication (a)$\Rightarrow$(c) 
of \prettyref{cor:existence_frechet_predual} and the observation that the strong dual of a Banach predual $(Y,\varphi)$ 
of $(X,\tau)$ is completely normable.

(c)$\Rightarrow$(b), (d)$\Rightarrow$(a), (e)$\Rightarrow$(a) These implications follow from the proof of the corresponding implications 
of \prettyref{cor:existence_cb_predual} and the observation that $(X_{\mathcal{B},\widetilde{\tau}}',\beta)$ 
is completely normable by \prettyref{prop:predual_complete} (c), 
which coincides with $(X,\gamma)_{b}'$.

(e)$\Leftrightarrow$(f) This equivalence follows from \prettyref{prop:BBC_CNC_normable}.
\end{proof}

\section{Linearisation}
\label{sect:linearisation}

In this section we study necessary and sufficient conditions for the existence of a strong linearisation of a 
bornological function space. 
If $X=\F$ is a space of $\K$-valued functions on a non-empty set $\Omega$ that satisfies the assumptions 
of \prettyref{thm:general_dixmier_ng} or \prettyref{cor:general_dixmier_ng_semi_reflexive} 
and we choose $Y\coloneqq \F_{\mathcal{B},\widetilde{\tau}}'$ and $T\coloneqq\mathcal{I}$, then we only need one additional ingredient 
to obtain a strong linearisation of $\F$ from the tuple $(\F_{\mathcal{B},\widetilde{\tau}}',\mathcal{I})$, namely a suitable map 
$\delta\colon \Omega\to \F_{\mathcal{B},\widetilde{\tau}}'$ which fulfils $\mathcal{I}(f)\circ\delta=f$ for every $f\in\F$. 
In order to fulfil $\mathcal{I}(f)\circ\delta=f$ for every $f\in\F$, i.e.
\[
f(x)=(\mathcal{I}(f)\circ\delta)(x)=\mathcal{I}(f)(\delta(x))=\delta(x)(f),\quad x\in\Omega,
\]
for every $f\in\F$, the map $\delta(x)$ has to be the point evaluation functional $\delta_{x}$ given by $\delta_{x}(f)\coloneqq f(x)$ 
for every $x\in\Omega$ and $f\in\F$. Thus we obtain a strong linearisation of $\F$ if 
$\delta_{x}\in\F_{\mathcal{B},\widetilde{\tau}}'$ for every $x\in\Omega$ and 
arrive at the following results.

\begin{cor}\label{cor:scb_linearisation}
Let $(\F,\tau)$ be a bornological locally convex Hausdorff space of $\K$-valued functions on a non-empty set $\Omega$ satisfying 
$\operatorname{(BBC)}$ and $\operatorname{(CNC)}$ for some $\widetilde{\tau}$ and $\mathcal{B}$ the family of $\tau$-bounded sets. 
If $\Delta(x)\coloneqq\delta_{x}\in\F_{\mathcal{B},\widetilde{\tau}}'$ for all $x\in\Omega$, then 
$(\Delta,\F_{\mathcal{B},\widetilde{\tau}}',\mathcal{I})$ is a strong complete barrelled linearisation of $\F$. 
\end{cor}

\begin{cor}\label{cor:scb_linearisation_semi_reflexive}
Let $(\F,\tau)$ be a bornological locally convex Hausdorff space of $\K$-valued functions on a non-empty set $\Omega$ 
satisfying $\operatorname{(BBCl)}$ and $\operatorname{(CNC)}$ for some 
$\widetilde{\tau}$ and $(\F,\widetilde{\gamma})$ a semi-reflexive space satisfying $\operatorname{(B\tau B}\widetilde{\gamma})$. 
If $\Delta(x)\coloneqq\delta_{x}\in(\F,\widetilde{\gamma})'=\F_{\mathcal{B},\widetilde{\tau}}'$ for all $x\in\Omega$, then 
$(\Delta,(\F,\widetilde{\gamma})',\mathcal{I})$ is a strong complete barrelled linearisation of $\F$.
\end{cor}

\begin{rem}\label{rem:point_eval_in_predual}
Let $(\F,\tau)$ be a locally convex Hausdorff space of $\K$-valued functions on a non-empty set $\Omega$, $\mathcal{B}$ the family of 
$\tau$-bounded sets and $\widetilde{\tau}$ a locally convex Hausdorff topology on $\F$. 
If $\tau_{\operatorname{p}}\leq\widetilde{\tau}$, then $\delta_{x}\in\F_{\mathcal{B},\widetilde{\tau}}'$ for all $x\in\Omega$. 
Indeed, for every $x\in\Omega$ we have $\delta_{x}\in(\F,\tau_{\operatorname{p}})'$ by definition and thus 
$\delta_{x}\in(\F,\widetilde{\tau})'\subset \F_{\mathcal{B},\widetilde{\tau}}'$ since 
$\tau_{\operatorname{p}}\leq\widetilde{\tau}$.
\end{rem}

\prettyref{cor:scb_linearisation} and \prettyref{rem:point_eval_in_predual} also give us a simple sufficient criterion when a bornological space $(\F,\tau)$ of continuous functions
has a separable predual, which generalises \cite[2.2 Remark, p.~870]{mujica1991}. 

\begin{prop}\label{prop:separable}
Let $(\F,\tau)$ be a locally convex Hausdorff space of $\K$-valued continuous 
functions on a non-empty separable topological Hausdorff space $\Omega$ satisfying 
$\operatorname{(BBC)}$ and $\operatorname{(CNC)}$ for some $\tau_{\operatorname{p}}\leq\widetilde{\tau}$ 
and $\mathcal{B}$ the family of $\tau$-bounded sets. 
Then $(\F_{\mathcal{B},\widetilde{\tau}}',\mathcal{I})$ is a complete barrelled predual of 
$(\F,\tau)$ and $(\F_{\mathcal{B},\widetilde{\tau}}',\beta)$ is separable. 
\end{prop}
\begin{proof}
$(\F_{\mathcal{B},\widetilde{\tau}}',\mathcal{I})$ is a complete barrelled predual of 
$(\F,\tau)$ by \prettyref{cor:scb_linearisation} and \prettyref{rem:point_eval_in_predual}. 
By the bipolar theorem the span of $\{\delta_{x}\;|\;x\in D\}$ is $\beta$-dense 
in $\F_{\mathcal{B},\widetilde{\tau}}'$ for any dense subspace $D$ of $\Omega$ since $\F$ is a space of continuous functions. This implies our statement on separability. 
\end{proof}

Next, we show that our sufficient conditions for the existence of a strong complete barrelled linearisation are also necessary. 

\begin{thm}\label{thm:existence_scb_linearisation}
Let $(\F,\tau)$ be a bornological locally convex Hausdorff space of $\K$-valued functions on a non-empty set $\Omega$ 
and $\mathcal{B}$ the family of $\tau$-bounded sets.
Then the following assertions are equivalent.
\begin{enumerate}
\item[(a)] $(\F,\tau)$ admits a strong complete barrelled linearisation.
\item[(b)] $(\F,\tau)$ has a semi-Montel prebidual $(Y,\varphi)$ such that $Y_{b}'$ is complete and for every $x\in\Omega$ 
there is a unique $y_{x}'\in Y'$ such that $\delta_{x}=\varphi(\cdot)(y_{x}')$.
\item[(c)] $(\F,\tau)$ satisfies $\operatorname{(BBC)}$ and $\operatorname{(CNC)}$ 
for some $\widetilde{\tau}$ such that $\tau_{\operatorname{p}}\leq\widetilde{\tau}$ and $(\F,\widetilde{\gamma})$ 
satisfies $\operatorname{(B\tau B}\widetilde{\gamma})$.
\item[(d)] $(\F,\tau)$ satisfies $\operatorname{(BBCl)}$ and $\operatorname{(CNC)}$ 
for some $\widetilde{\tau}$ such that $\tau_{\operatorname{p}}\leq\widetilde{\tau}$ 
and $(\F,\widetilde{\gamma})$ is semi-reflexive and satisfies $\operatorname{(B\tau B}\widetilde{\gamma})$.
\item[(e)] $(\F,\tau)$ satisfies $\operatorname{(BBC)}$ and $\operatorname{(CNC)}$ for some $\widetilde{\tau}$ such that $\tau_{\operatorname{p}}\leq\widetilde{\tau}$.
\item[(f)] $(\F,\tau)$ satisfies $\operatorname{(BBC)}$ and $\operatorname{(CNC)}$ for some $\widetilde{\tau}$ such that 
$\delta_{x}\in \F_{\mathcal{B},\widetilde{\tau}}'$ for all $x\in\Omega$.
\end{enumerate}
\end{thm}
\begin{proof}
(b)$\Rightarrow$(a) By the proof of \prettyref{cor:existence_cb_predual} the tuple $(Y_{b}',\varphi)$ is a complete barrelled 
predual of $(\F,\tau)$. We set $\delta\colon\Omega\to Y'$, $\delta(x)\coloneqq y_{x}'$. Then we have
\[
(\varphi(f)\circ\delta)(x)=\varphi(f)(y_{x}')=\delta_{x}(f)=f(x)
\]
for all $f\in\F$ and $x\in\Omega$. Hence $(\delta,Y_{b}',\varphi)$ is a strong complete barrelled linearisation of $\F$. 

(a)$\Rightarrow$(c) Since $(\F,\tau)$ admits a strong complete barrelled linearisation, there are a 
complete barrelled locally convex Hausdorff space $Y$, a map $\delta\colon\Omega\to Y$ and a topological isomorphism $T\colon(\F,\tau)\to Y_{b}'$ such that $T(f)\circ\delta=f$ for all $f\in\F$. 
Then $(\F,\tau)$ satisfies $\operatorname{(BBC)}$ and $\operatorname{(CNC)}$ for $\widetilde{\tau}\coloneqq\sigma_{T}(\F,Y)$ and 
$(\F,\widetilde{\gamma})$ satisfies $\operatorname{(B\tau B}\widetilde{\gamma})$ by \prettyref{prop:predual_to_BBC_CNC} (a). 
From $T(f)(\delta(x))=f(x)$ for all $f\in\F$ and $x\in\Omega$, we deduce that $\tau_{\operatorname{p}}$ and $\sigma_{T}(\F,Y_{0})$ coincide 
on $\F$ where $Y_{0}$ denotes the span of $\{\delta(x)\;|\;x\in\Omega\}$ which is dense in $Y$ 
by \prettyref{prop:equivalent_def_of_linearisation_dense}, and $\sigma_{T}(\F,Y_{0})$ is defined 
by the system of seminorms 
\[
p_{N}(f)\coloneqq\sup_{y\in N}|T(f)(y)|,\quad f\in \F,
\]
for finite sets $N\subset Y_{0}$. Hence we have $\tau_{\operatorname{p}}=\sigma_{T}(\F,Y_{0})\leq\sigma_{T}(\F,Y)=\widetilde{\tau}$.

(c)$\Rightarrow$(b) Due to \prettyref{cor:semi_montel_gen_mixed} $(\F,\widetilde{\gamma})$ is a semi-Montel space and 
thus semi-reflexive. It follows from \prettyref{cor:scb_linearisation_semi_reflexive} 
and \prettyref{rem:point_eval_in_predual} with $Y\coloneqq (\F,\widetilde{\gamma})$ and $\varphi\coloneqq\mathcal{I}$ 
that $(Y,\varphi)$ is a semi-Montel prebidual of $(\F,\tau)$ such that $Y_{b}'$ is a complete barrelled space and 
$\delta_{x}=\varphi(\cdot)(\delta_{x})$ with $\delta_{x}\in Y'$. Suppose that for $x\in\Omega$ there is another $y_{x}'\in Y'$ 
such that $\delta_{x}=\varphi(\cdot)(y_{x}')$. This implies that 
$\Phi_{\varphi}(y_{x}')=\Phi_{\varphi}(\delta_{x})$ for the map $\Phi_{\varphi}\colon Y_{b}'\to (\F,\tau)_{b}'$ 
from \prettyref{prop:predual_into_dual} and thus $y_{x}'=\delta_{x}$ by \prettyref{prop:predual_into_dual}. 

(c)$\Rightarrow$(d) This implication follows from the proof of \prettyref{cor:existence_cb_predual}.

(d)$\Rightarrow$(a) This implication follows from \prettyref{cor:scb_linearisation_semi_reflexive} 
and \prettyref{rem:point_eval_in_predual} with $Y\coloneqq (\F,\widetilde{\gamma})$ and $\varphi\coloneqq\mathcal{I}$.

(c)$\Rightarrow$(e) This implication is obvious. 

(e)$\Rightarrow$(f) We only need to show that $\delta_{x}\in \F_{\mathcal{B},\widetilde{\tau}}'$ for all $x\in\Omega$, which is a 
consequence of \prettyref{rem:point_eval_in_predual}.

(f)$\Rightarrow$(a) This implication follows from \prettyref{cor:scb_linearisation}.
\end{proof}

\begin{cor}\label{cor:existence_scbDF_linearisation}
Let $(\F,\tau)$ be a bornological locally convex Hausdorff space of $\K$-valued functions on a non-empty set $\Omega$ 
and $\mathcal{B}$ the family of $\tau$-bounded sets.
Then the following assertions are equivalent.
\begin{enumerate}
\item[(a)] $(\F,\tau)$ admits a strong complete barrelled DF-linearisation.
\item[(b)] $(\F,\tau)$ has a semi-Montel prebidual $(Y,\varphi)$ such that $Y_{b}'$ is a complete DF-space and for every $x\in\Omega$ 
there is a unique $y_{x}'\in Y'$ such that $\delta_{x}=\varphi(\cdot)(y_{x}')$.
\item[(c)] $(\F,\tau)$ is a Fr\'echet space satisfying $\operatorname{(BBC)}$ and $\operatorname{(CNC)}$ 
for some $\widetilde{\tau}$ such that $\tau_{\operatorname{p}}\leq\widetilde{\tau}$ 
and $(\F,\widetilde{\gamma})$ satisfies $\operatorname{(B\tau B}\widetilde{\gamma})$.
\item[(d)] $(\F,\tau)$ is a Fr\'echet space satisfying $\operatorname{(BBCl)}$ and $\operatorname{(CNC)}$ 
for some $\widetilde{\tau}$ such that $\tau_{\operatorname{p}}\leq\widetilde{\tau}$ and $(\F,\widetilde{\gamma})$ 
is semi-reflexive and satisfies $\operatorname{(B\tau B}\widetilde{\gamma})$.
\item[(e)] $(\F,\tau)$ is a Fr\'echet space satisfying $\operatorname{(BBC)}$ and $\operatorname{(CNC)}$ for some $\widetilde{\tau}$ such that $\tau_{\operatorname{p}}\leq\widetilde{\tau}$.
\item[(f)] $(\F,\tau)$ is a Fr\'echet space satisfying $\operatorname{(BBC)}$ and $\operatorname{(CNC)}$ for some $\widetilde{\tau}$ such that $\delta_{x}\in \F_{\mathcal{B},\widetilde{\tau}}'$ for all $x\in\Omega$.
\end{enumerate}
\end{cor}
\begin{proof}
This statement follows from the proofs of \prettyref{thm:existence_scb_linearisation} and \prettyref{cor:existence_cbDF_predual}. 
\end{proof}

\begin{cor}\label{cor:existence_sF_linearisation}
Let $(\F,\tau)$ be a bornological locally convex Hausdorff space of $\K$-valued functions on a non-empty set $\Omega$ 
and $\mathcal{B}$ the family of $\tau$-bounded sets.
Then the following assertions are equivalent.
\begin{enumerate}
\item[(a)] $(\F,\tau)$ admits a strong Fr\'echet linearisation.
\item[(b)] $(\F,\tau)$ has a semi-Montel prebidual $(Y,\varphi)$ such that $Y_{b}'$ is a Fr\'echet space and for every $x\in\Omega$ 
there is a unique $y_{x}'\in Y'$ such that $\delta_{x}=\varphi(\cdot)(y_{x}')$.
\item[(c)] $(\F,\tau)$ is a complete DF-space satisfying $\operatorname{(BBC)}$ and $\operatorname{(CNC)}$ 
for some $\widetilde{\tau}$ such that $\tau_{\operatorname{p}}\leq\widetilde{\tau}$ and $(\F,\gamma)$ satisfies 
$\operatorname{(B\tau B}\gamma)$.
\item[(d)] $(\F,\tau)$ is a complete DF-space satisfying $\operatorname{(BBCl)}$ and $\operatorname{(CNC)}$ for some $\widetilde{\tau}$ 
such that $\tau_{\operatorname{p}}\leq\widetilde{\tau}\leq\tau$ and $(\F,\gamma)$ is semi-reflexive. 
\item[(e)] $(\F,\tau)$ is a complete DF-space satisfying $\operatorname{(BBC)}$ and $\operatorname{(CNC)}$ 
for some $\widetilde{\tau}$ such that $\tau_{\operatorname{p}}\leq\widetilde{\tau}$. 
\item[(f)] $(\F,\tau)$ is a complete DF-space satisfying $\operatorname{(BBC)}$ and $\operatorname{(CNC)}$ for some $\widetilde{\tau}$ such that $\delta_{x}\in \F_{\mathcal{B},\widetilde{\tau}}'$ for all $x\in\Omega$.
\item[(g)] $(\F,\tau)$ is a complete DF-space satisfying $\operatorname{(BBC)}$ for some $\widetilde{\tau}$ such that 
$\tau_{\operatorname{p}}\leq\widetilde{\tau}$ and the Fr\'echet space $(\F_{\mathcal{B},\widetilde{\tau}}',\beta)$ is distinguished. 
\item[(h)] $(\F,\tau)$ is a complete DF-space satisfying $\operatorname{(BBC)}$ for $\tau_{\operatorname{p}}$ 
and the Fr\'echet space $(\F_{\mathcal{B},\tau_{\operatorname{p}}}',\beta)$ is distinguished. 
\end{enumerate}
\end{cor}
\begin{proof}
The first seven equivalences follow from the proofs of \prettyref{thm:existence_scb_linearisation} 
and \prettyref{cor:existence_frechet_predual} where one uses for the implication (g)$\Rightarrow$(a) 
in addition \prettyref{rem:point_eval_in_predual}. 

(g)$\Rightarrow$(h) If $(\F,\tau)$ satisfies $\operatorname{(BBC)}$ for some $\widetilde{\tau}$ such that 
$\tau_{\operatorname{p}}\leq\widetilde{\tau}$, then it satisfies  $\operatorname{(BBC)}$ for $\tau_{\operatorname{p}}$ 
by \prettyref{rem:BBC_CNC} (d). Further, we have $\F_{\mathcal{B},\tau_{\operatorname{p}}}'=\F_{\mathcal{B},\widetilde{\tau}}'$ 
and $\beta_{\mathcal{B},\tau_{\operatorname{p}}}=\beta_{\mathcal{B},\widetilde{\tau}}=\beta$ 
by \prettyref{rem:predual_independent}, implying that 
$(\F_{\mathcal{B},\tau_{\operatorname{p}}}',\beta)$ 
is distinguished.

(h)$\Rightarrow$(g) This implication is obvious. 
\end{proof}

\begin{cor}\label{cor:existence_sB_linearisation}
Let $(\F,\tau)$ be a bornological locally convex Hausdorff space of $\K$-valued functions on a non-empty set $\Omega$ 
and $\mathcal{B}$ the family of $\tau$-bounded sets.
Then the following assertions are equivalent.
\begin{enumerate}
\item[(a)] $(\F,\tau)$ admits a strong Banach linearisation.
\item[(b)] $(\F,\tau)$ has a semi-Montel prebidual $(Y,\varphi)$ such that $Y_{b}'$ is completely normable and for every $x\in\Omega$ 
there is a unique $y_{x}'\in Y'$ such that $\delta_{x}=\varphi(\cdot)(y_{x}')$.
\item[(c)] $(\F,\tau)$ is a completely normable space satisfying $\operatorname{(BBC)}$ and $\operatorname{(CNC)}$ 
for some $\widetilde{\tau}$ such that $\tau_{\operatorname{p}}\leq\widetilde{\tau}$ and $(\F,\gamma)$ 
satisfies $\operatorname{(B\tau B}\gamma)$.
\item[(d)] $(\F,\tau)$ is a completely normable space satisfying $\operatorname{(BBCl)}$ 
for some $\widetilde{\tau}$ such that $\tau_{\operatorname{p}}\leq\widetilde{\tau}\leq\tau$ 
and $(\F,\gamma)$ is semi-reflexive. 
\item[(e)] $(\F,\tau)$ is a completely normable space satisfying $\operatorname{(BBC)}$ for some $\widetilde{\tau}$ 
such that $\tau_{\operatorname{p}}\leq\widetilde{\tau}$.
\item[(f)] There is a norm $\vertiii{\cdot}$ on $\F$ which induces $\tau$ such that $B_{\vertiii{\cdot}}$ is $\widetilde{\tau}$-compact for some $\widetilde{\tau}$ such that $\tau_{\operatorname{p}}\leq\widetilde{\tau}$.
\item[(g)] There is a norm $\vertiii{\cdot}$ on $\F$ which induces $\tau$ such that $B_{\vertiii{\cdot}}$ is $\widetilde{\tau}$-compact for some $\widetilde{\tau}$ such that $\delta_{x}\in \F_{\mathcal{B},\widetilde{\tau}}'$ for all $x\in\Omega$.
\item[(h)] $(\F,\tau)$ is a completely normable space satisfying $\operatorname{(BBC)}$ for $\tau_{\operatorname{p}}$.
\item[(i)] There is a norm $\vertiii{\cdot}$ on $\F$ which induces $\tau$ such that $B_{\vertiii{\cdot}}$ is 
$\tau_{\operatorname{p}}$-compact.
\end{enumerate}
\end{cor}
\begin{proof}
The first seven equivalences follow from the proofs of \prettyref{thm:existence_scb_linearisation} 
and \prettyref{cor:existence_banach_predual}. 

(e)$\Rightarrow$(h) This implication follows from \prettyref{rem:BBC_CNC} (d) with 
$\widetilde{\tau}_{0}\coloneqq \tau_{\operatorname{p}}$.

(h)$\Rightarrow$(e) This implication is obvious.

(h)$\Leftrightarrow$(i) This equivalence follows from \prettyref{prop:BBC_CNC_normable} 
with $\widetilde{\tau}\coloneqq \tau_{\operatorname{p}}$. 
\end{proof}

We close this section with a characterisation of continuous strong linearisations. 
We call a topological space $\Omega$ a $gk_{\R}$\emph{-space} 
if for any completely regular space $Y$ and any map $f\colon\Omega\to Y$, 
whose restriction to each compact $K\subset\Omega$ is continuous, the map is already continuous on 
$\Omega$. If a $gk_{\R}$-space $\Omega$ is also completely regular, then it is called a  
$k_{\R}$\emph{-space} (see \cite[(2.3.7) Proposition, p.~22]{buchwalter1969}). 
Examples of Hausdorff $gk_{\R}$-spaces are Hausdorff $k$-spaces by 
\cite[3.3.21 Theorem, p.~152]{engelking1989}. Examples of Hausdorff $k_{\R}$-spaces are 
metrisable spaces by \cite[Proposition 11.5, p.~181]{james1999} and 
\cite[3.3.20 Theorem, p.~152]{engelking1989},  locally compact Hausdorff spaces 
and strong duals of Fr\'echet--Montel spaces (\emph{DFM-spaces}) 
by \cite[Proposition 3.27, p.~95]{fabian2011} and \cite[4.11 Theorem (5), p.~39]{kriegl1997}. 
The underlying idea of the proof of our next result comes from \cite[Theorem 2.1, p.~187]{jaramillo2009}.

\begin{prop}\label{prop:point_eval_kR}
Let $(\F,\tau)$ be a locally convex Hausdorff space of $\K$-valued continuous functions on a 
non-empty Hausdorff $gk_{\R}$-space $\Omega$ and $\mathcal{B}$ the family of $\tau$-bounded sets. 
If $(\F,\tau)$ satisfies $\operatorname{(BBC)}$ 
for some $\widetilde{\tau}$ such that $\tau_{\operatorname{co}}\leq\widetilde{\tau}$, then the map 
\[
\Delta\colon \Omega\to 
(\F_{\mathcal{B},\widetilde{\tau}}',\beta),\; 
\Delta(x)\coloneqq\delta_{x},
\] 
is continuous. 
\end{prop}
\begin{proof}
First, we note that the map $\Delta$ is well-defined by \prettyref{rem:point_eval_in_predual} because 
$\tau_{\operatorname{p}}\leq\tau_{\operatorname{co}}\leq\widetilde{\tau}$. 
Since $\Omega$ is a $gk_{\R}$-space and the locally convex Hausdorff space $(\F_{\mathcal{B},\widetilde{\tau}}',\beta)$ completely regular by \cite[Proposition 3.27, p.~95]{fabian2011}, 
we only need to show that the restricted map $\Delta_{\mid K}$ is 
continuous for every compact set $K\subset\Omega$. 
Let $B\in\mathcal{B}$ be absolutely convex and $\widetilde{\tau}$-compact and $K\subset\Omega$ 
a compact set. Due to \prettyref{rem:BBC_CNC} (d) $B$ is also $\tau_{\operatorname{co}}$-compact.
The restriction 
$(\F,\tau_{\operatorname{co}})\to (\mathcal{C}(K),\|\cdot\|_{K}),\;f\mapsto f_{\mid K}$,
is continuous where $\|g\|_{K}\coloneqq \sup_{x\in K}|g(x)|$ for all $g\in \mathcal{C}(K)$. 
This implies that $B_{\mid K}\coloneqq\{f_{\mid K}\;|\;f\in B\}$ is $\|\cdot\|_{K}$-compact. Due 
to \cite[Theorem 47.1 (Ascoli's theorem) (b), p.~290]{munkres2000} we obtain that $B_{\mid K}$ is equicontinuous. Hence  for every $\varepsilon>0$ and $x\in K$ there is a neighbourhood 
$U(x)\subset K$ of $x$ such that for all $y\in U(x)$ we have 
\[
\sup_{f\in B}|f(x)-f(y)|\leq\varepsilon,
\]
which implies 
\[
\sup_{f\in B}|\Delta(x)(f)-\Delta(y)(f)|=\sup_{f\in B}|f(x)-f(y)|\leq\varepsilon.
\]
Therefore $\Delta_{\mid K}$ is continuous, which closes the proof.  
\end{proof}

\begin{thm}\label{thm:existence_scb_cont_linearisation}
Let $(\F,\tau)$ be a bornological locally convex Hausdorff space of $\K$-valued continuous functions on a non-empty 
topological Hausdorff space $\Omega$ and $\mathcal{B}$ the family of $\tau$-bounded sets.
Consider the following assertions.
\begin{enumerate}
\item[(a)] $(\F,\tau)$ admits a continuous strong complete barrelled linearisation.
\item[(b)] $(\F,\tau)$ satisfies $\operatorname{(BBC)}$ and $\operatorname{(CNC)}$ for some $\widetilde{\tau}$ such that $\tau_{\operatorname{p}}\leq\widetilde{\tau}$, and every $B\in\mathcal{B}$ 
is equicontinuous.
\item[(c)] $(\F,\tau)$ satisfies $\operatorname{(BBC)}$ and $\operatorname{(CNC)}$ for some $\widetilde{\tau}$ such that 
\[
\Delta\colon \Omega\to 
(\F_{\mathcal{B},\widetilde{\tau}}',\beta),\;\Delta(x)\coloneqq\delta_{x},
\] 
is a well-defined continuous map.
\item[(d)] $(\F,\tau)$ satisfies $\operatorname{(BBC)}$ and $\operatorname{(CNC)}$ for some $\widetilde{\tau}$ such that 
$\tau_{\operatorname{co}}\leq\widetilde{\tau}$.
\end{enumerate}
Then it holds (a)$\Leftrightarrow$(b)$\Leftrightarrow$(c). If $\Omega$ is a $gk_{\R}$-space, then it holds (d)$\Rightarrow$(c).
\end{thm}
\begin{proof}
(a)$\Rightarrow$(b) Due to \prettyref{thm:existence_scb_linearisation} we only need to show that 
every $B\in\mathcal{B}$ is equicontinuous. Since $(\F,\tau)$ admits a continuous strong complete barrelled linearisation, there are a 
complete barrelled locally convex Hausdorff space $Y$, a continuous map 
$\delta\colon\Omega\to Y$ and a topological isomorphism $T\colon(\F,\tau)\to Y_{b}'$ such that $T(f)\circ\delta=f$ 
for all $f\in\F$. 
For every $B\in\mathcal{B}$ we have 
\[
 \sup_{f\in B}|f(x)-f(y)|
=\sup_{f\in B}|T(f)(\delta(x))-T(f)(\delta(y))|
=\sup_{f\in B}|\Phi_{T}(\delta(x))(f)-\Phi_{T}(\delta(y))(f)|
\]
for all $x,y\in\Omega$ with $\Phi_{T}$ from \prettyref{prop:predual_into_dual} for $X\coloneqq\F$. 
From the continuity of $\delta$ and \prettyref{prop:predual_into_dual} 
we deduce that every $B\in\mathcal{B}$ is equicontinuous. 

(b)$\Rightarrow$(c) Due to \prettyref{thm:existence_scb_linearisation} we only need to show that $\Delta$ is continuous. 
This directly follows from the observation that
\[
\sup_{f\in B}|\Delta(x)(f)-\Delta(y)(f)|=\sup_{f\in B}|f(x)-f(y)|.
\]
for all $B\in\mathcal{B}$ and $x,y\in\Omega$ and the equicontinuity of every $B\in\mathcal{B}$.

(c)$\Rightarrow$(a) Due to \prettyref{cor:scb_linearisation} $(\Delta,\F_{\mathcal{B},\widetilde{\tau}}',\mathcal{I})$ 
is a strong complete barrelled linearisation of $\F$. Since $\Delta$ is continuous, this triple is also a continuous linearisation. 

(d)$\Rightarrow$(c) if $\Omega$ a $gk_{\R}$-space. This implication follows from \prettyref{cor:scb_linearisation} and 
\prettyref{prop:point_eval_kR}.
\end{proof}

Using \prettyref{cor:existence_scbDF_linearisation}, we may add further statements to 
\prettyref{thm:existence_scb_cont_linearisation} that are equivalent to statement (a). 
In the same way we may prove a corresponding characterisation of continuous strong complete barrelled DF-linearisation by using 
\prettyref{cor:existence_scbDF_linearisation}, of continuous strong Fr\'echet linearisations 
by using \prettyref{cor:existence_sF_linearisation} and of continuous strong Banach linearisations 
by using \prettyref{cor:existence_sB_linearisation}. 
In the case of continuous strong Banach linearisations we get 
\cite[Theorem 2.2, Corollary 2.3, p.~188--189]{jaramillo2009} back with an improvement of 
\cite[Corollary 2.3, p.~189]{jaramillo2009} from Hausdorff $k$-spaces $U=\Omega$ to Hausdorff $gk_{\R}$-spaces. 

\begin{exa}\label{ex:scb_linearisation_CP}
Let $\Omega\subset\R^{d}$ be open and $P(\partial)$ a hypoelliptic linear partial differential operator on $\mathcal{C}^{\infty}(\Omega)$. 

(i) Let $\mathcal{V}$ be a point-detecting directed family of continuous weights. 
If the space $(\mathcal{C}_{P}\mathcal{V}(\Omega),\tau_{\mathcal{V}})$ is bornological, then $(\Delta,\mathcal{C}_{P}\mathcal{V}(\Omega)_{\mathcal{B},\tau_{\operatorname{co}}}',\mathcal{I})$ is a continuous strong complete barrelled 
linearisation of $\mathcal{C}_{P}\mathcal{V}(\Omega)$ by \prettyref{ex:borno_frechet_hypo} and the proof 
of the implication (d)$\Rightarrow$(c) of \prettyref{thm:existence_scb_cont_linearisation} since 
$\Omega$ is metrisable and thus a Hausdorff $k_{\R}$-space.
If $\mathcal{V}$ is countable and increasing, then $(\Delta,\mathcal{C}_{P}\mathcal{V}(\Omega)_{\mathcal{B},\tau_{\operatorname{co}}}',\mathcal{I})$ is a continuous strong complete barrelled DF-linearisation of 
$\mathcal{C}_{P}\mathcal{V}(\Omega)$, and if $\mathcal{V}=\{v\}$, then $(\Delta,\mathcal{C}_{P}v(\Omega)_{\mathcal{B},\tau_{\operatorname{co}}}',\mathcal{I})$ is a continuous strong Banach linearisation of 
$\mathcal{C}_{P}v(\Omega)$. 

(ii) Let $\mathcal{V}\coloneqq (v_{n})_{n\in\N}$ be a drecreasing, regularly decreasing family of continuous functions $v_{n}\colon\Omega\to(0,\infty)$. 
Then $(\Delta,\mathcal{VC}_{P}(\Omega)_{\mathcal{B},\tau_{\operatorname{co}}}',\mathcal{I})$ is 
a continuous strong Fr\'echet linearisation of $\mathcal{VC}_{P}(\Omega)$
by \prettyref{ex:borno_DF_hypo}, the proof of the implication 
(f)$\Rightarrow$(a) of \prettyref{cor:existence_frechet_predual} and \prettyref{prop:point_eval_kR}.
\end{exa}

\bibliography{biblio_linearisation_existence}
\bibliographystyle{plainnat}

\end{document}